\documentclass[a4paper,twoside,10pt]{article}

\usepackage[english]{babel}
\usepackage[utf8]{inputenc}
\usepackage{lmodern,color}
\usepackage[T1]{fontenc}
\usepackage{indentfirst}
\usepackage{amsmath,amssymb}
\usepackage[percent]{overpic}
\usepackage[a4paper,top=3cm,bottom=3cm,left=3cm,right=3cm,%
bindingoffset=5mm]{geometry}
\usepackage{lmodern}
\usepackage[T1]{fontenc}
\usepackage{lettrine}
\usepackage{booktabs}
\usepackage{subcaption}
\usepackage{authblk}
\usepackage{emp}
\usepackage{amsthm}
\usepackage{amsmath,amssymb,amsthm}
\usepackage{braket}
\usepackage{chngpage}
\usepackage{cases}
\usepackage{graphicx}
\usepackage{epstopdf}
\usepackage{tabularx}
\usepackage{multirow}
\usepackage{enumerate}
\usepackage{todonotes}

\newcommand{\numberset}{\mathbb}
\newcommand{\N}{\numberset{N}}
\newcommand{\R}{\numberset{R}}

\newcommand{\B}{\numberset{B}}
\newcommand{\Pk}{\numberset{P}}

\renewcommand{\epsilon}{\varepsilon}
\renewcommand{\theta}{\vartheta}
\renewcommand{\rho}{\varrho}
\renewcommand{\phi}{\varphi}

\newcommand{\kk}{\boldsymbol{k}}

\newcommand{\gr}{\nabla}

\def\kint{k_o}
\def\kbou{k_{\partial}}
\def\hpartialE{{h_{\partial E}}}
\def\hbulkE{{h_{E}}}
\def\hpartial{{h_{\partial}}}
\def\hbulk{{h}}
\def\Vem{V^h_{\kk}(E)}
\def\VemG{V^h_{\kk}}
\def\Pkint{\Pk_{\kint}}

\def\hh{\mathcal{H}}
\def\P0{{\Pi^{0, E}_{\kint-2}}}
\def\emixed{e_{u}}

\def\PN{{\Pi^{\nabla, E}_{\kint}}}

\def\PP0{{\boldsymbol{\Pi}^{0, E}_{k-1}}}

\def\Stab{\mathcal{S}^E}
\def\vint{\delta_o}
\def\vp{\delta_{\partial}}

{\left\lbrace\begin{array}{@{}l@{}}}%
{\end{array}\right.}
\theoremstyle{definition}

\theoremstyle{remark}
\newtheorem{remark}{Remark}[section]

\theoremstyle{remark}

\theoremstyle{plain}

\newtheorem{proposition}{Proposition}[section]
\newtheorem{corollary}{Corollary}[section]
\newtheorem{lemma}{Lemma}[section]
\newtheorem{definition}{Definition}[section]
\usepackage{lipsum}

\usepackage{authblk}
\usepackage{color}

\usepackage{setspace}

\author[1]{L. Beir\~ao da Veiga \thanks{lourenco.beirao@unimib.it}}

\author[1]{G. Vacca \thanks{giuseppe.vacca@unimib.it}}

\affil[1]{Dipartimento di Matematica e Applicazioni,
Universit\`a degli Studi di Milano Bicocca,
Via Roberto Cozzi 55 - 20125 Milano, Italy}

\title{\textbf{Sharper error estimates for Virtual Elements and a bubble-enriched version}}

\date{\today}

\begin{document}

\maketitle

\begin{abstract}
In the present contribution we develop a sharper error analysis for the Virtual Element Method, applied to a model elliptic problem, that separates the element boundary and element interior contributions to the error. As a consequence we are able to propose a variant of the scheme that allows to take advantage of polygons with many edges (such as those composing Voronoi meshes or generated by agglomeration procedures) in order to yield a more accurate discrete solution. The theoretical results are supported by numerical experiments.
\end{abstract}

\section{Introduction}
\label{sec:intro}

The Virtual Element Method (VEM) was introduced in \cite{volley,autostoppisti} as a generalization of the
finite element method that is able to cope with general polytopal meshes, even with non-convex
and badly-shaped elements. Since its introduction, the VEM enjoyed a large success in the
numerical analysis and engineering communities, with many papers devoted to develop its
theoretical foundations and many others devoted to applications in different areas (a short representative list includes
\cite{
brezzi-marini:2013,
berrone:2016,
AMV:2018,
artioli:2019,
visinoni:2020,
sacco:2020,
BPP:2017,
cangiani:2017,
zhang:2019,
mora:2020,
scacchi:2017,
fumagalli:2018,
sukumar:2019,
CGS:2017,
CGS:2018,
MPP:2018,
mascotto:2019,
BRV:2019,
reddy:2017,
paulino:2020,
cao-chen:2018,
brenner-sung:2019,
GMV:2019,
frittelli:2018,
natarajan:2020}). 
This contribution falls into the first category, as it originates from a natural question about Virtual Elements (which is often heard at conferences) and it improves the existing theoretical results; on the basis of our findings, we also propose an interesting variant of the scheme. Our investigation focuses on a model 2D elliptic problem.

In the present manuscript we investigate if, and how, the presence of many edges can help the approximation capabilities of the method. Indeed, standard $H^1$-conforming virtual elements have degrees of freedom associated to element edges and vertexes (in addition to moments inside). Therefore one may wonder if, given a certain element size (diameter), having many edges may help somehow the interpolation accuracy of the discrete space, and if this will reflect also on the final error among the discrete and exact solutions. Basically, the answer is no, but the investigation allows to shed more light on the matter and develop an interesting variant.

Looking into the interpolation capabilities of the VEM space, by a refined analysis we show
that the $H^1$ interpolation error $\| u - u_I \|_{H^1(E)}$ on each element (polygon) $E$ can be split into a boundary contribution and a bulk contribution. 
Assuming a sufficiently regular target function, the boundary contribution behaves as $h_{\partial E}^k$ (with $h_{\partial E}$ denoting the maximum edge length and $k$ representing the VEM ``polynomial'' degree) and therefore it decreases in the presence of smaller edges. On the contrary, the bulk part behaves as $h_{E}^k$, with $\hbulkE$ the element diameter. Therefore, basically, having more edges does not help as the second term will dominate the error.
On the other hand, this investigation leads to the following idea: if one increases the
degree of the VEM only inside the element (that in practice corresponds to adding DoFs inside,
which can then be statically condensed) then the bulk approximation order improves.
For such ``enriched'' VEM, elements with small edges indeed lead to more
accurate interpolation. Moreover, having a richer internal DoFs set (more moments) allows to
compute projections on polynomials of higher order, and thus guarantees also a more accurate
approximation of the bilinear form. As a consequence, the above enhanced accuracy
directly reflects also on the ``consistency'' error among the discrete and the continuous formulations. 

A further important ingredient in our analysis is investigating the stability properties of the discrete problem. The most widely adopted VEM stabilization in the literature, that is the so called ``dofi-dofi'' stabilization \cite{volley}, is not robust with respect to the number of edges. This is the reason why, even in the deep analysis of \cite{BLR:2017,brenner-sung:2018,chen-huang:2018}, a uniform bound on the number of edges is assumed for the dofi-dofi stabilization. Since such assumption would represent a strong limitation to the scopes of the present study, we develop an improved stabilility investigation that leads to a sharper bound in terms of the number of edges. By a careful use of the discrete interpolant and a suitable bound for the element $H^{1/2}$ boundary norm, we are finally able to show  an error estimate of the kind
$$
\|u - u_h\|_{1, \Omega}^2  \lesssim \alpha \sum_{E \in \Omega_h} 
\Big( (\alpha + \ell_E)^{1/2} \, h_E^{\kint} + h_{\partial E}^{\kbou} \Big)^2 \, ,
$$
where $\kint \ge \kbou$ are respectively the internal and boundary degrees, $\ell_E$ denotes the number of edges of element $E$, and $\alpha$ is a logarithmic term (thus essentially negligible) of the maximum ratio among the larger and the smaller edge of each element.
In our study we also investigate another well-know stabilization form, the so-called ``trace'' stabilization \cite{wriggers:2016} which leads to a final error that is fully robust also with respect to the number (and different size) of edges:
$$
\|u - u_h\|_{1, \Omega}^2  \lesssim \sum_{E \in \Omega_h} 
\Big( h_E^{\kint} + h_{\partial E}^{\kbou} \Big)^2  \, .
$$
Our theoretical results are supported by a set of numerical tests, where we can appreciate from the practical standpoint the two distinct contributions to the error (boundary and bulk), and the improvement of the enriched version. The numerical experiments are developed both for quadrilateral/Voronoi meshes with edge subdivision and on meshes generated by an agglomeration procedure.

The paper is organized as follows. 
In Section \ref{sec:notations} we present the continuous problem, we fix some notations and discuss the mesh assumptions. Afterwards, in Section \ref{sec:VEM} we introduce the generalized VEM and investigate the stability properties of the scheme. In Section \ref{sec:convergence} we develop the interpolation and convergence properties of the method. In Section \ref{sec:tests} we present the numerical experiments. In the Appendix we show the proof of a useful Lemma.


\section{Notations and Preliminaries}
\label{sec:notations}
Throughout the paper, we will follow the usual notation for Sobolev spaces
and norms \cite{Adams:1975}.
Hence, for an open bounded domain $\omega$,
the norms in the spaces $W^s_p(\omega)$ and $L^p(\omega)$ are denoted by
$\|{\cdot}\|_{W^s_p(\omega)}$ and $\|{\cdot}\|_{L^p(\omega)}$ respectively.
Norm and seminorm in $H^{s}(\omega)$ are denoted respectively by
$\|{\cdot}\|_{s,\omega}$ and $|{\cdot}|_{s,\omega}$,
while $(\cdot,\cdot)_{\omega}$ and $\|\cdot\|_{\omega}$ denote the $L^2$-inner product and the $L^2$-norm (the subscript $\omega$ may be omitted when $\omega$ is the whole computational
domain $\Omega$).

\subsection{Continuous Problem}
\label{sub:cp}

In the present paper for simplicity  we consider the Poisson equation, but observe that the same approach can be easily extended to more general problems.

Let $\Omega \subset \R^2$ be the computational domain and let $f\in L^2(\Omega)$ represent the external load, then our model problem reads
\begin{equation}
\label{eq:poisson-c}
\left \{
\begin{aligned}
& \text{find $u \in V$ s.t.} 
\\
& a(u,  v) = ( f,  v) \qquad \text{for all $v \in V$,}
\end{aligned}
\right.
\end{equation}
where $V=H^1_0(\Omega)$ and $a(\cdot,  \cdot) \colon V \times V \to \R$ is given by
\begin{equation}
\label{eq:a-c}
a(u,  v) := \int_{\Omega} \nabla u \cdot \nabla v \, {\rm d}\Omega
\qquad \text{for all $u, v \in V$.}
\end{equation}
It is well known that Equation \eqref{eq:poisson-c} has a unique solution $u \in V$ s.t.
$|u|_{1, \Omega} \leq \|f\|_{V^*}$.


\subsection{Mesh notations and assumptions}
\label{sub:notations}
From now on, we will denote with $E$ a general polygon having $\ell_E$ edges,  $e$ will denote a general edge of $E$ and $\partial E := \cup_{i=1}^{\ell_E} e_i$. 
Let us introduce the following notation:
\[
\begin{gathered}
\hbulkE := {\rm diameter}(E) \,,
\quad  
h_e := \text{length}(e) \,,
\quad 
\hpartialE := \max_{e \in \partial E} h_{e} \,,
\quad 
\hh_E := \frac{\max_{e \in \partial E} h_{e}}{\min_{e \in \partial E} h_{e}} \,.
\end{gathered}
\]
Let $\set{\Omega_h}_h$ be a sequence of decompositions of $\Omega$ into general polygons $E$, 
where we set
\begin{equation}
\label{eq:h}
\hbulk := \sup_{E \in \Omega_h} \hbulkE \,,
\qquad \qquad
\hpartial := \sup_{E \in \Omega_h} \hpartialE \,,
\qquad \qquad
\hh := \sup_{E \in \Omega_h} \hh_E \,.
\end{equation}
We suppose that $\set{\Omega_h}_h$ fulfils the following assumption ~\cite{BLR:2017,brenner-guan-sung:2017,brenner-sung:2018,chen-huang:2018}:
\begin{itemize}
\item[\textbf{(A1)}]
there exists a uniform positive constant $\rho$ such that $E \in \set{\Omega_h}_h$ is star-shaped with respect to a ball $B_E$ of radius $ \geq\, \rho \, \hbulkE$.
\end{itemize}
Note that in the present paper we do not require any condition in order to forbid ``small edges'' (that is, edges of a generic element may be arbitrarily smaller than its diameter) or a uniform bound on the number of edges. We will instead investigate explicitly the influence of such parameters in our estimates. In this respect, we introduce the following definition.
\begin{definition}\label{pqu}
Let $\{ {\cal T}_h \}_h$ represent a family of one-dimensional grids, each meshing a bounded interval $I^h \subset {\mathbb R}$. Then, such family is denoted as {\bf piecewise quasi-uniform} if there exist ${\overline m} \in {\mathbb N}$ and $\overline{c} \in {\mathbb R}^+$ such that the following holds.
Any mesh in the family can be decomposed into at most $\overline{m}$ disjoint subset grids (each meshing a sub-interval of $I^h$), each of them being quasi-uniform (precisely, the ratio among the largest and the smallest element of each subset mesh is bounded by $\overline{c}$).
\end{definition}
We now note that, for each element $E$ of $\{ \Omega_h \}_h$, the partition induced by the edges on $\partial E$ can be naturally interpreted as a one dimensional mesh. More precisely, fix any vertex ${\boldsymbol \nu}$ of $E$ and denote by $\Gamma_E \colon [0,|\partial E|] \to \partial E$ the unique curvilinear abscissae parametrization of $\partial E$ with counterclockwise orientation that satisfies $\Gamma_E(0)=\Gamma_E(|\partial E|)={\boldsymbol \nu}$. Then, the push-backward of the edges $e \subset \partial E$ constitute a partition of the interval $[0,|\partial E|]$, which is what we call the one dimensional mesh induced by the edges on $\partial E$. Roughly, this is nothing but the one dimensional mesh obtained by ``unwrapping'' the boundary of $E$ into an interval of the real line.
We can now introduce the following assumption on $\{\Omega_h\}_h$. 
\begin{itemize}
\item [\textbf{(A2)}] The family of one-dimensional meshes induced on each mesh element boundary $\partial E$ by its edges, $E \in \{\Omega_h\}_h$, is piecewise quasi-uniform.
\end{itemize}
The above assumption covers essentially all cases of interest; it allows for a number of edges per element that does not need to be uniformly bounded, and allows also the presence of ``small edges'' (in the sense described above). Mesh families created by agglomeration, cracking, gluing, etc..  of existing meshes are, for instance, included.  Some example in shown in Fig. \ref{fig:PQU}. While assumption \textbf{(A1)} will be required through all the paper, assumption \textbf{(A2)} will be needed only for certain stabilizations.
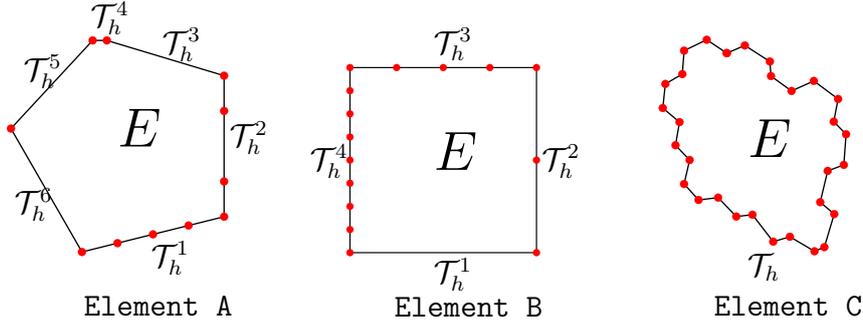
\begin{figure}[!htb]
\begin{center}
\vspace{0.5cm}
\begin{overpic}[scale=0.2]{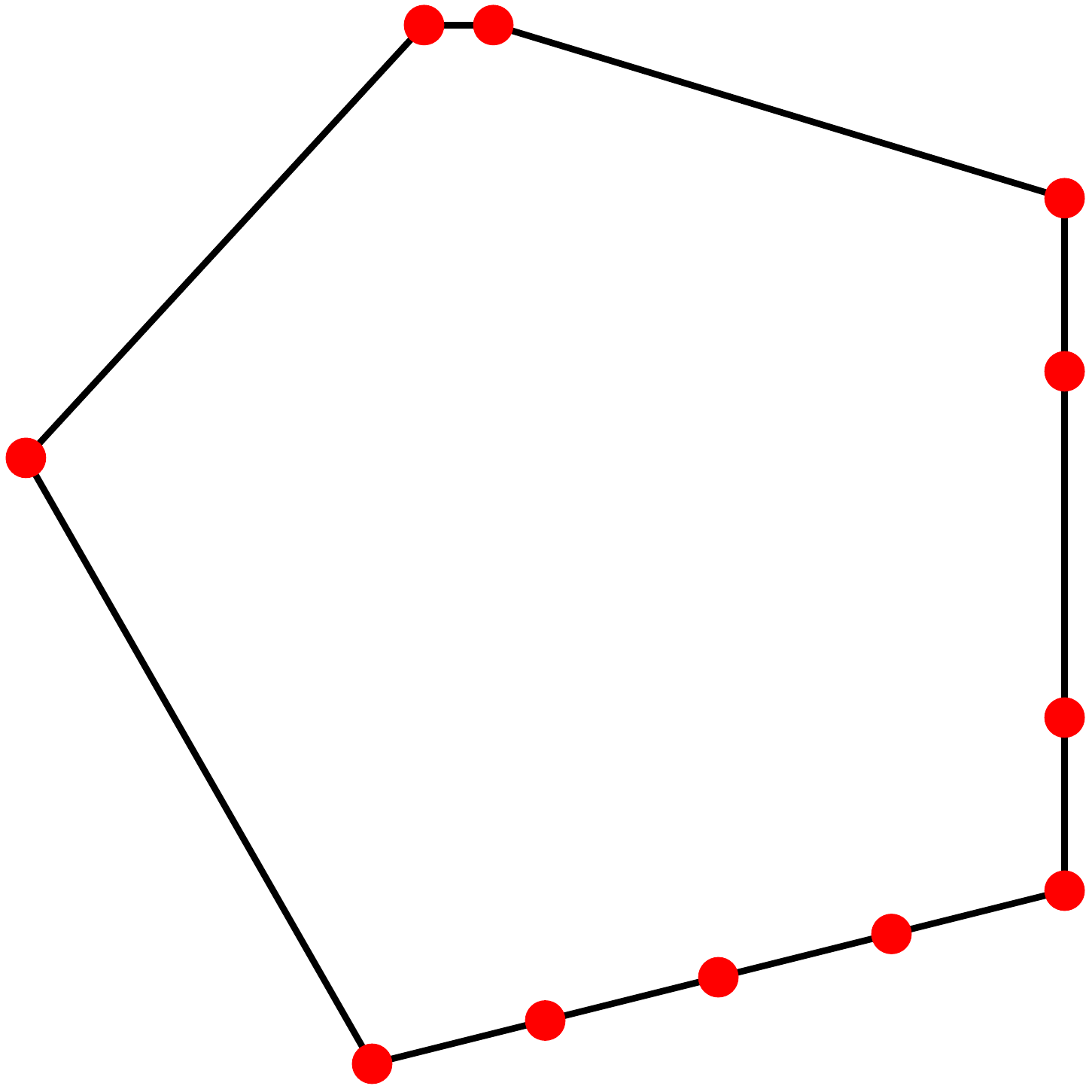}
\put(50,50){\huge{{$E$}}}
\put(65,-4){\large{{$\mathcal{T}^1_h$}}}
\put(100,50){\large{{$\mathcal{T}^2_h$}}}
\put(70,92){\large{{$\mathcal{T}^3_h$}}}
\put(38,104){\large{{$\mathcal{T}^4_h$}}}
\put(8,80){\large{{$\mathcal{T}^5_h$}}}
\put(4,20){\large{{$\mathcal{T}^6_h$}}}
\put(35,-26){\large{{\texttt{Element A}}}}
\end{overpic}
\qquad \qquad 
\begin{overpic}[scale=0.175]{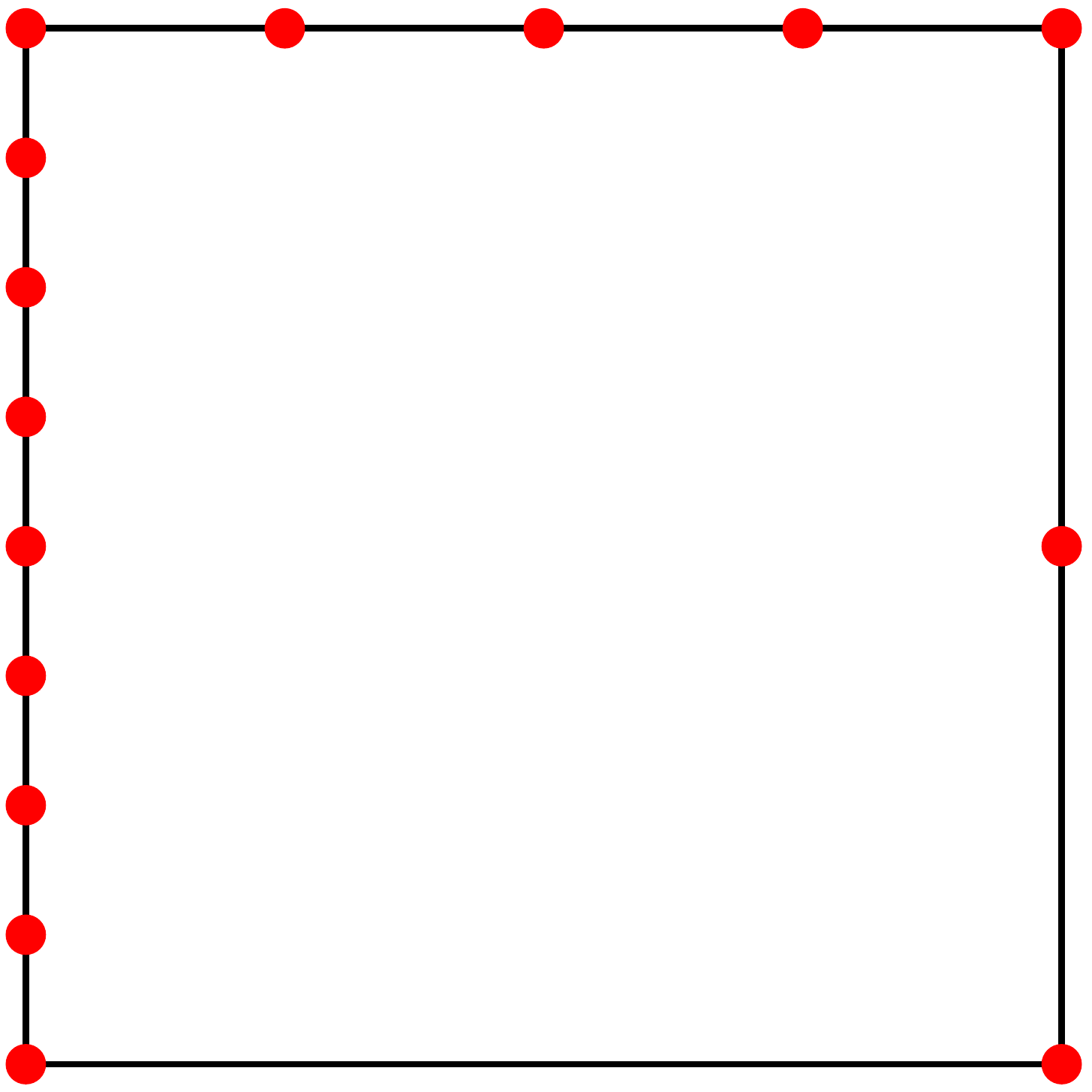}
\put(45,45){\huge{{$E$}}}
\put(45,-13){\large{{$\mathcal{T}^1_h$}}}
\put(100,45){\large{{$\mathcal{T}^2_h$}}}
\put(45,105){\large{{$\mathcal{T}^3_h$}}}
\put(-18,45){\large{{$\mathcal{T}^4_h$}}}
\put(25,-30){\large{{\texttt{Element B}}}}
\end{overpic}
\qquad \qquad 
\begin{overpic}[scale=0.2]{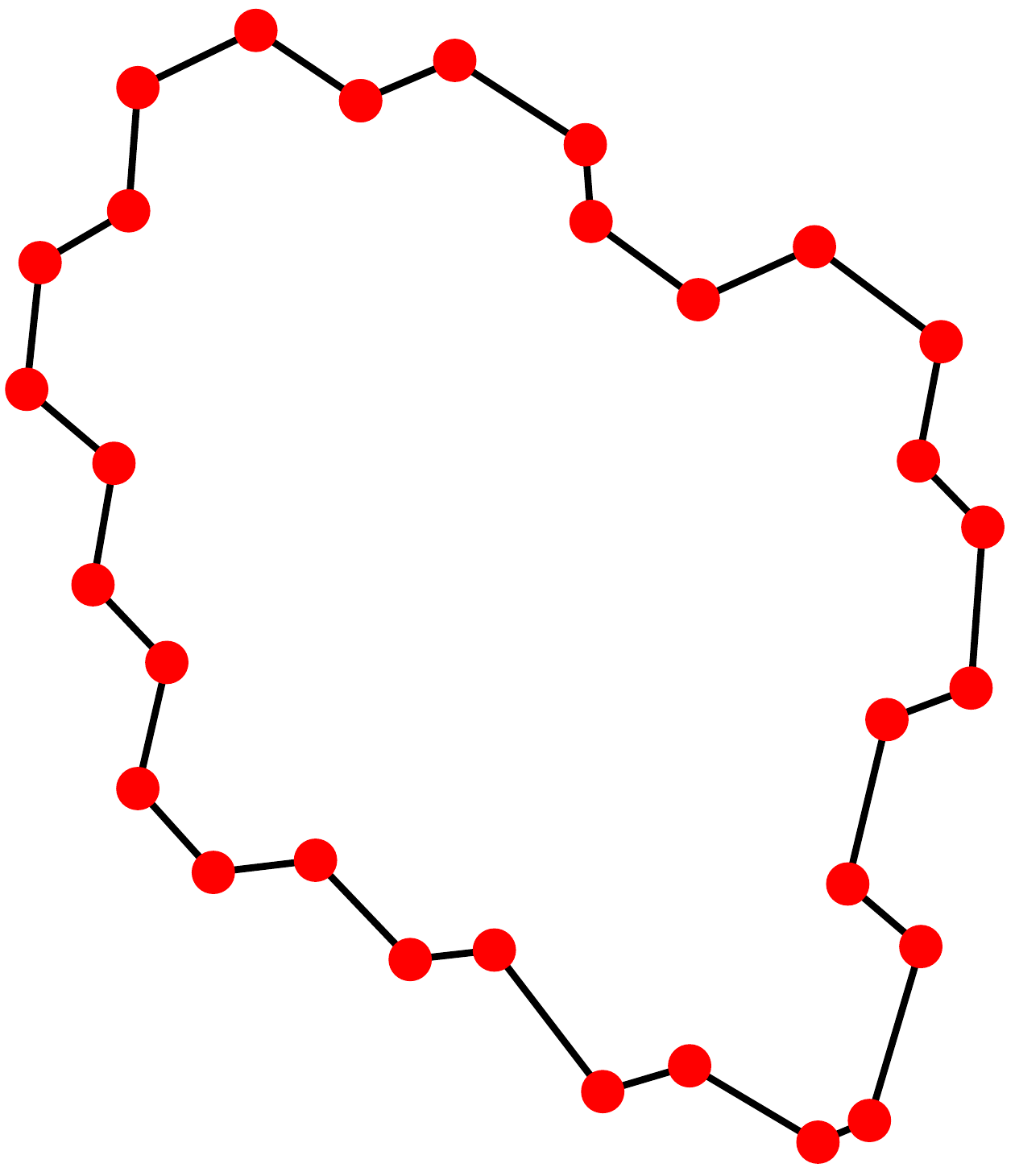}
\put(40,45){\huge{{$E$}}}
\put(40,-8){\large{{$\mathcal{T}_h$}}}
\put(25,-26){\large{{\texttt{Element C}}}}
\end{overpic}
\vspace{1cm}
\caption{
\texttt{Element A}, \texttt{Element B}, \texttt{Element C} represent mesh refinements/types that satisfy \textbf{(A2)}.
\texttt{Element A}: polygon with a ``small edge''.
\texttt{Element B}: polygon with edges with length $h_e = 2^{-kn}$ with $k=0,1,2,3$ and $n \in \N^+$.
\texttt{Element C}: polygon arising from an agglomeration procedure.
}
\label{fig:PQU}
\end{center}
\end{figure}

Using standard VEM notations, for $n \in \N$, $s \in {\mathbb R}^{+}$,   and for any $E \in \Omega_h$,  let us introduce the spaces:
\begin{itemize}
\item $\Pk_n(\omega)$ the set of polynomials on $\omega$ of degree $\leq n$  (with $\Pk_{-1}(\omega)=\{ 0 \}$),
\item $\B_n(\partial E) := \{v \in C^0(\partial E) \quad \text{s.t} \quad v_{|e} \in \Pk_n(e) \quad \text{for all edge $e \subset \partial E$} \}$,
\item $\Pk_n(\Omega_h) := \{q \in L^2(\Omega) \quad \text{s.t} \quad q_{|E} \in  \Pk_n(E) \quad \text{for all $E \in \Omega_h$}\}$,
\item $H^s(\Omega_h) := \{v \in L^2(\Omega) \quad \text{s.t} \quad v_{|E} \in  H^s(E) \quad \text{for all $E \in \Omega_h$}\}$ equipped with the broken norm and seminorm
\[
\|v\|^2_{s,\Omega_h} := \sum_{E \in \Omega_h} \|v\|^2_{s,E}\,,
\qquad 
|v|^2_{s,\Omega_h} := \sum_{E \in \Omega_h} |v|^2_{s,E} \,,
\]
\end{itemize}
and the following polynomial projections:
\begin{itemize}
\item the $\boldsymbol{L^2}$\textbf{-projection} $\Pi_n^{0, E} \colon L^2(E) \to \Pk_n(E)$, given by
\begin{equation}
\label{eq:P0_k^E}
\int_Eq_n (v - \, {\Pi}_{n}^{0, E}  v) \, {\rm d} E = 0 \qquad  \text{for all $v \in L^2(E)$  and $q_n \in \Pk_n(E)$,} 
\end{equation}

\item the $\boldsymbol{H^1}$\textbf{-seminorm projection} ${\Pi}_{n}^{\nabla,E} \colon H^1(E) \to \Pk_n(E)$, defined by 
\begin{equation}
\label{eq:Pn_k^E}
\left\{
\begin{aligned}
& \int_E \gr  \,q_n \cdot \gr ( v - \, {\Pi}_{n}^{\nabla,E}   v)\, {\rm d} E = 0 \quad  \text{for all $v \in H^1(E)$ and  $q_n \in \Pk_n(E)$,} \\
& \int_{\partial E}(v - \,  {\Pi}_{n}^{\nabla, E}  v) \, {\rm d}s= 0 \, .
\end{aligned}
\right.
\end{equation}
\end{itemize}

In the following the symbol $\lesssim$ will denote a bound up to a generic positive constant, independent of the quantities $\hbulkE$, $\hpartialE$, $\hbulk$, $\hpartial$ and $\ell_E$ but which may depend on  $\Omega$, on the ``polynomial'' order $\kk$ (introduced below) and on the regularity constants appearing in the adopted assumptions (that is \textbf{(A1)}, \textbf{(A2)} or none).

\section{Generalized Virtual Elements}
\label{sec:VEM}

Let $k \geq 1$. For any $E \in \Omega_h$ the \textbf{standard local virtual element space} \cite{volley} is given by
\begin{equation}
\label{eq:space-s}
V^h_k(E) := \left\{
v_h \in H^1(E) \quad \text{s.t.} \quad v_h|_{\partial E} \in \B_{k}(\partial E), \quad -\Delta v_h \in \Pk_{k-2}(E) 
\right\} \,.
\end{equation}
The idea now is to decouple the polynomial order on the boundary and in the bulk of the element.  Let $\kint$ and $\kbou$ be two positive integers with 
$\kint \geq \kbou$ and let $\kk = (\kint, \, \kbou)$. Note that, although $\kint=\kbou$ is admissible in the following theory, the most interesting case for the present study is $\kint > \kbou$.
For any $E \in \Omega_h$ we define the \textbf{generalized local virtual element space}: 
%
\begin{equation}
\label{eq:space}
\Vem := \left\{
v_h \in H^1(E) \quad \text{s.t.} \quad v_h|_{\partial E} \in \B_{\kbou}(\partial E), \quad -\Delta v_h \in \Pk_{\kint -2}(E) 
\right\} \,.
\end{equation}
%
Using standard tools in VEM literature \cite{volley,autostoppisti} it can be proved that the space $\Vem$  satisfies the following properties:
\begin{itemize}
\item [\textbf{(P1)}] \textbf{Polynomial space inclusion:} $\Pk_{\kbou} \subseteq \Vem$ but in general $\Pkint \not \subseteq \Vem$;
\item [\textbf{(P2)}] \textbf{VEM spaces inclusions:} 
$V^h_{\kbou}(E) \subseteq \Vem \subseteq V^h_{\kint}(E)$;
%
%
%
\item [\textbf{(P3)}] \textbf{Degrees of Freedom:} the following linear operators $\mathbf{D_V}$
(see Figure \ref{fig:VEMdofs}) constitute a set of DoFs for $\Vem$:
\begin{itemize}
\item[$\mathbf{D_V1}$]:  the values of $v_h$ at the vertexes of the polygon $E$,
\item[$\mathbf{D_V2}$]: the values of $v_h$ at $\kbou-1$ distinct points of every edge $e \in \partial E$,
\item[$\mathbf{D_V3}$]: the moments of $v_h$ against a polynomial basis  $\{m_i\}_i$ of $\Pk_{\kint-2}(E)$  with $\|m_i\|_{L^{\infty}(E)} = 1$:
$$
\frac{1}{|E|}\int_E v_h \, m_{i} \, {\rm d}E  \,;
$$
\end{itemize}
\item [\textbf{(P4)}] \textbf{Projections:}
the DoFs $\mathbf{D_V}$ allow us to compute exactly 
\[
\PN \colon \Vem \to \Pkint(E), \qquad
\P0 \colon \Vem \to \Pk_{\kint-2}(E) \,.
\]
\end{itemize}
\begin{figure}[!htb]
\begin{center}
\begin{overpic}[scale=0.16]{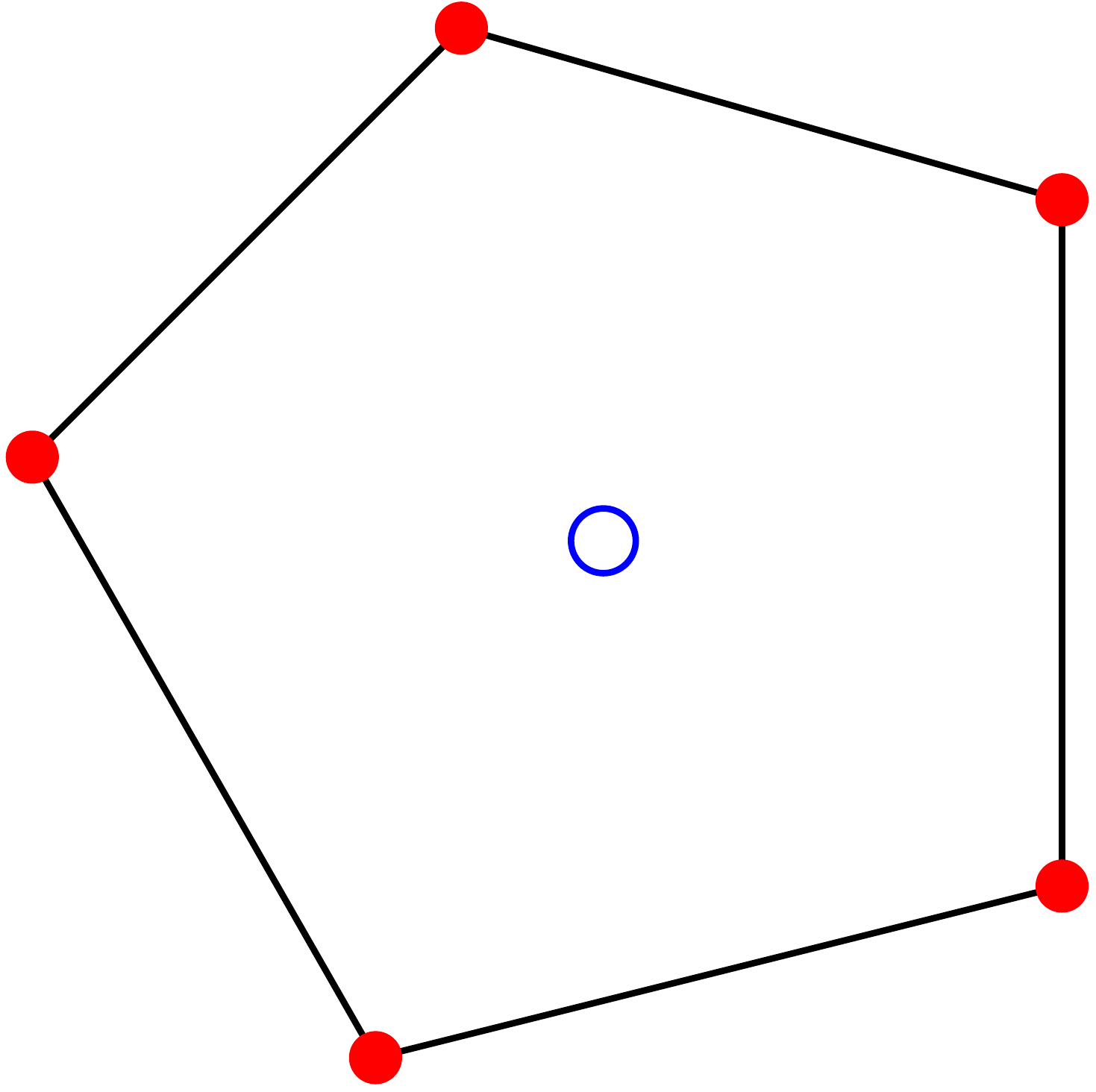}
\put(20,-14) {{$\kint =2$, $\kbou =1$}}
\end{overpic}
\quad
\begin{overpic}[scale=0.16]{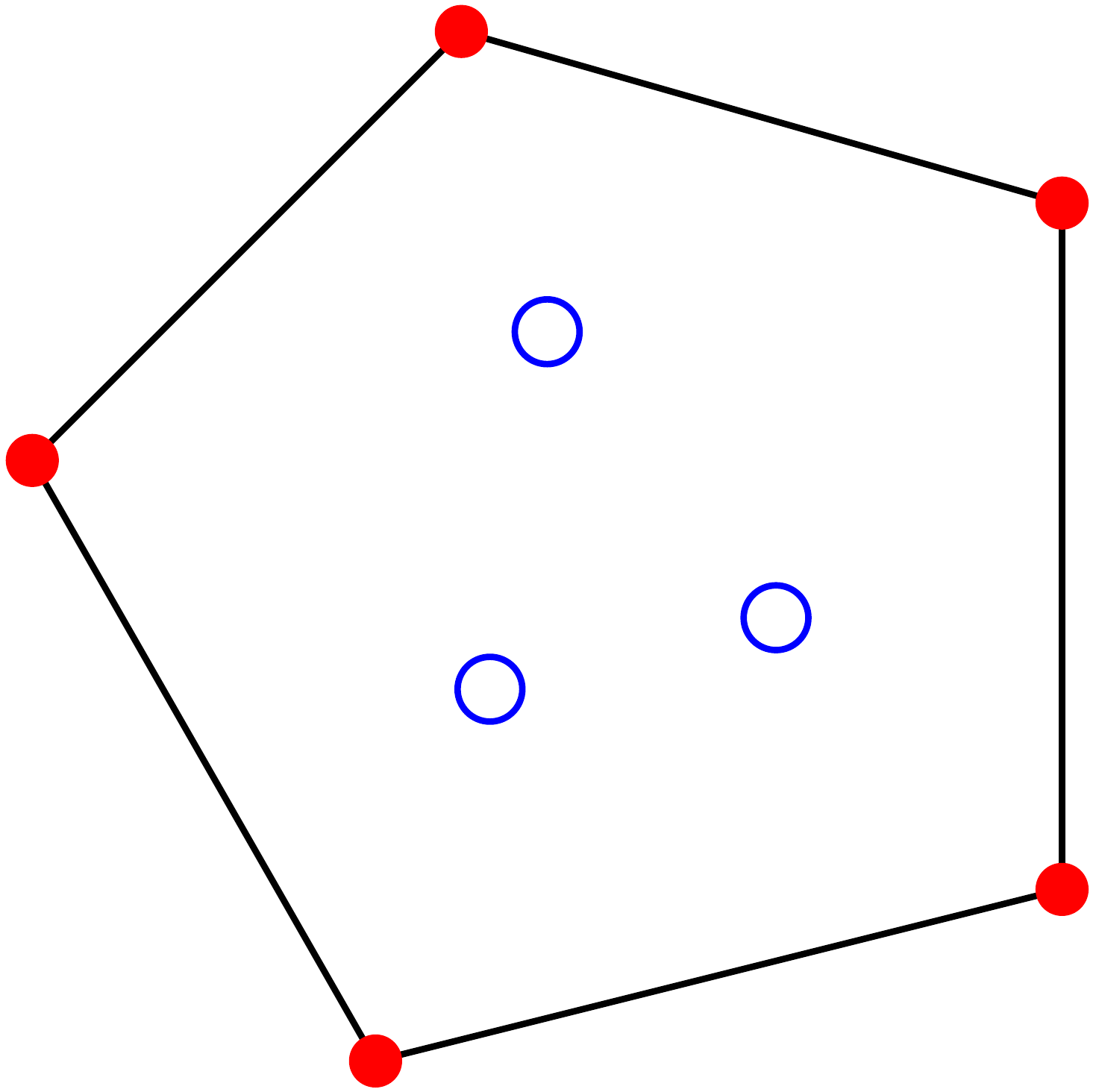}
\put(20,-14) {{$\kint =3$, $\kbou =1$}}
\end{overpic}
\quad
\begin{overpic}[scale=0.16]{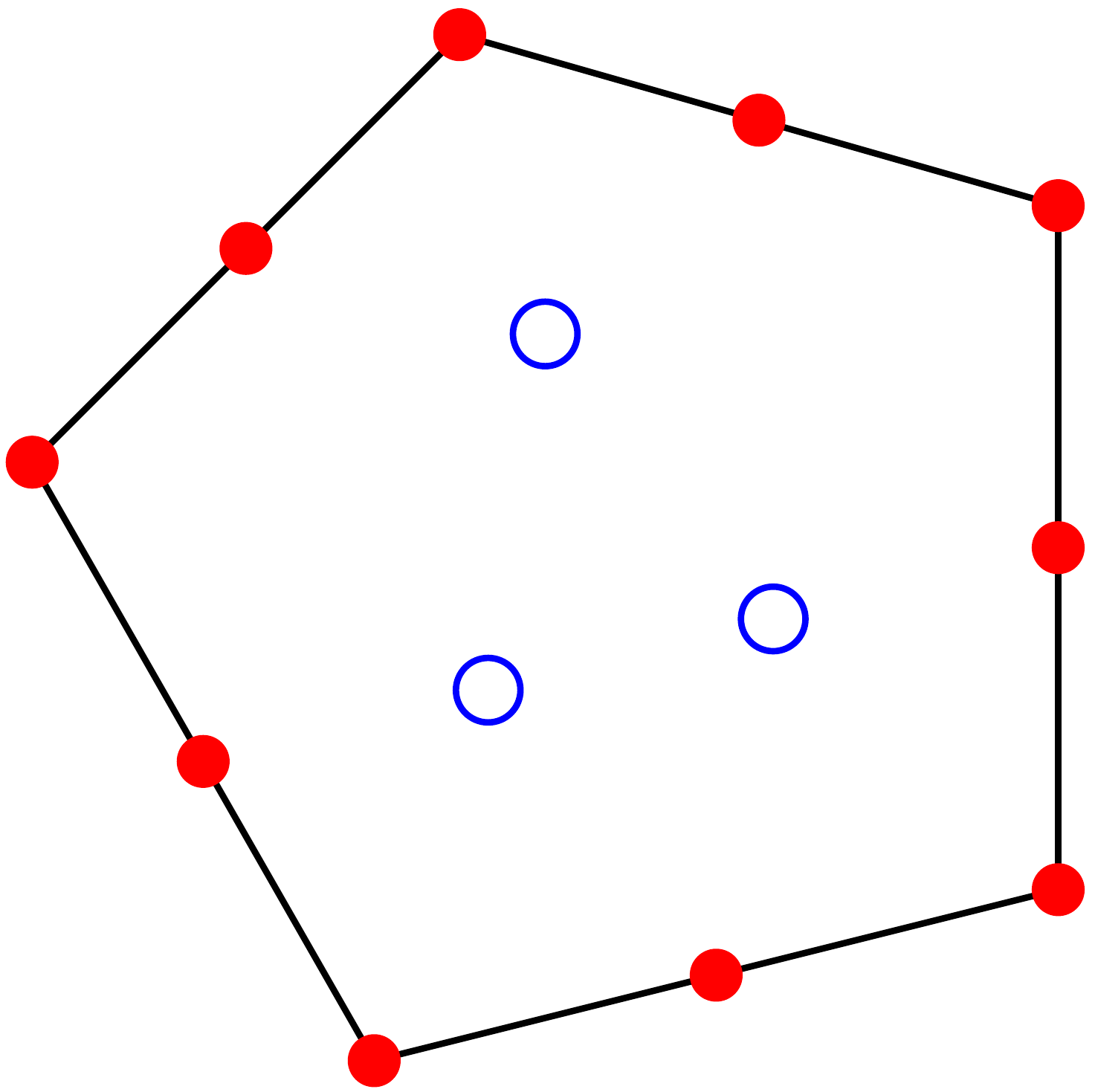}
\put(20,-14) {{$\kint =3$, $\kbou =2$}}
\end{overpic}
\quad
\begin{overpic}[scale=0.16]{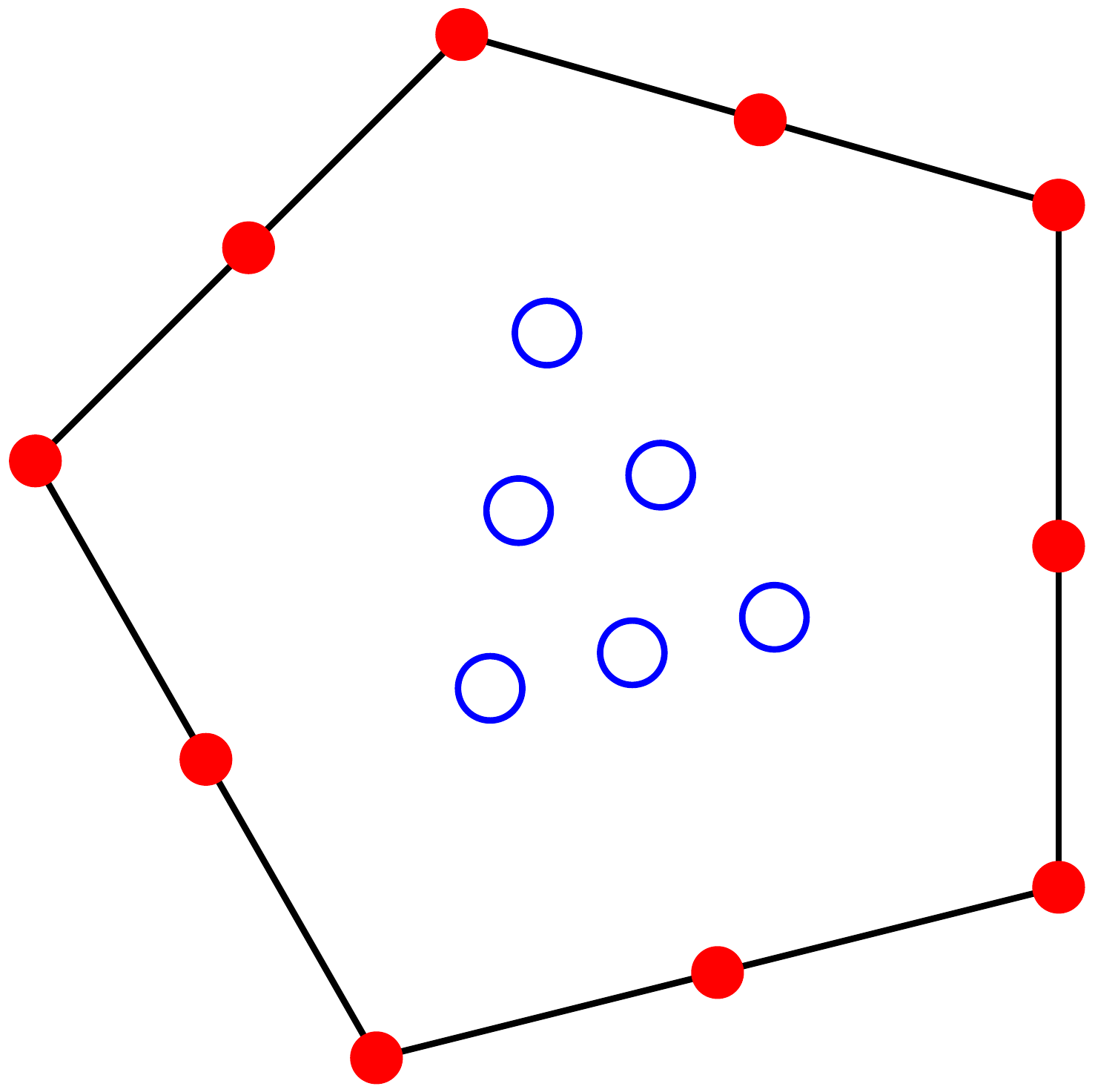}
\put(20,-14) {{$\kint =4$, $\kbou =2$}}
\end{overpic}
\end{center}
\vspace{2ex}
\caption{Example of DoFs for different values of $\kint$ and $\kbou$.}
\label{fig:VEMdofs}
\end{figure}


\begin{remark}
\label{rm:enhanced}
Using the same procedure in \cite{projectors} it would be possible to define the 
``enhanced'' version of the space $\Vem$ such that the ``full'' $L^2$-projection $\Pi_{\kint}^{0,E}$ is computable by the DoFs.
\end{remark}

We define the global virtual element space as
\begin{equation}
\label{eq:space-g}
\VemG := \left\{ v_h \in V  \quad \text{s.t.} \quad v_h|_E \in \Vem \quad \text{for all $E \in \Omega_h$} \right\} \,,
\end{equation}
with the obvious associated sets of global degrees of freedom. 

Finally we remark that the internal degrees of freedom $\mathbf{D_V3}$ can be eliminated from the final linear system by a static condensation procedure, and therefore are much cheaper (form the computational perspective) than the boundary ones.

\subsection{Discrete bilinear forms and load term approximation}
\label{sub:forms}

The next step in the construction of our method is to define a discrete version of the gradient-gradient form $a(\cdot,  \cdot)$ in \eqref{eq:a-c}.
First of all we decompose into local contributions the bilinear form by defining
\[
a(u,  v) =: \sum_{E \in \Omega_h} a^E(u,  v) \,.
\]
It is clear that for an arbitrary pair $(u_h,  v_h) \in \Vem \times \Vem$, the quantity $a^E(u_h,  v_h)$ is not computable since $u_h$ and $v_h$ are not known in closed form.
Therefore, following the usual procedure in the VEM setting, we introduce an approximated discrete bilinear form.
Exploiting the property \textbf{(P4)} and recalling \textbf{(P1)}, let \begin{equation}
\label{eq:ah-E}
a_h^E(\cdot,  \cdot) \colon [\Vem + \Pkint(E)] \times  [\Vem + \Pkint(E)] \to
\R 
\end{equation}
be a computable approximation of the continuous form $a^E(\cdot, \cdot)$ 
defined  by
 \begin{equation}
\label{eq:ah-E-def}
a_h^E(u_h,  v_h) :=
a^E(\PN u_h, \, \PN v_h) +
\Stab((I - \PN) u_h, \, (I - \PN) v_h)
\end{equation}
for all $u_h$, $v_h \in [\Vem + \Pkint(E)]$.
There are different choices for the symmetric stabilizing bilinear form
\begin{equation}
\label{eq:S-E}
\Stab(\cdot,  \cdot) \colon [\Vem + \Pkint(E)] \times  [\Vem + \Pkint(E)] \to
\R \, .
\end{equation}
Noticing that $[ \Vem + \Pkint(E) ] \subseteq V^h_{\kint}(E)$
we here focus on the following two classical ones:
\begin{itemize}
\item \texttt{dofi-dofi} stabilization \cite{volley}:
let $\vec{u}_h$, $\vec{v}_h$ denote the real valued vectors containing the values of the local degrees of freedom associated to $u_h$, $v_h$ in the enlarged space $V^h_{\kint}(E)$ (that correspond to $\mathbf{D_V1}$, $\mathbf{D_V2}$ and $\mathbf{D_V3}$ with $\kbou$ taken equal to $\kint$), then
\begin{equation}
\label{eq:dofidofi}
\Stab_{\texttt{d}} (u_h, v_h) =  \vec{u}_h \cdot  \vec{v}_h \,,
\end{equation}
\item \texttt{trace} stabilization \cite{wriggers:2016}: let $\partial_s v_h$ denote the tangential derivative of $v_h$ along $\partial E$, then
\begin{equation}
\label{eq:trace}
\Stab_{\texttt{t}} (u_h, v_h) =  
 h_E \,  \int_{\partial E} \partial_s u_h \, \partial_s v_h \, {\rm d}s  \,.
\end{equation}
\end{itemize}

\noindent
The global approximated bilinear form $a_h(\cdot,  \cdot)$
is defined by simply summing the local contributions:
\begin{equation}
\label{eq:a_h}
a_h(u_h,  v_h) := \sum_{E \in \Omega_h} a_h^E(u_h,  v_h) 
\qquad \text{for all $u_h, v_h \in [\VemG + \Pkint(\Omega_h)]$.}
\end{equation}
It is  straightforward to check that the bilinear form $a_h(\cdot, \, \cdot)$ satisfies the following:
\begin{itemize}
\item $\boldsymbol{\kint}$\textbf{-consistency property}:
for all $q_{\kint} \in \Pkint(\Omega_h)$ and  $v_h \in [\VemG + \Pkint(\Omega_h)]$
\begin{equation}
\label{eq:a_h-cons}
a_h(q_{\kint},  v_h) = a(q_{\kint},  v_h) \, .
\end{equation} 
\end{itemize}
Concerning the approximation of the right-hand side $( f,  v)$ in \eqref{eq:poisson-c}, we define the approximated load $f_h \in \Pk_{\kint-2}(\Omega_h)$ given by (for $\kint \ge 2$)
\begin{equation}
\label{eq:fh}
{f_h}|_E := \P0 f \qquad \text{for all $E \in \Omega_h$,}
\end{equation}
and define the computable right-hand side
\begin{equation}
\label{eq:right}
(f_h,  v_h) := 
\left \{
\begin{aligned}
&\sum_{E \in \Omega_h} (f_h,  v_h)_E \qquad & \text{for $\kint \geq 2$,}
\\
&\sum_{E \in \Omega_h} \int_E f  {\rm d}E \, \frac{1}{|\partial E|} \int_{\partial E} v_h \, {\rm d}s \qquad & \text{for $\kint = 1$.}
\end{aligned}
\right.
\end{equation}

\subsection{Coercivity of the bilinear form}
\label{sec:stab}

In this section we study the coercivity property of the bilinear form $a_h^E(\cdot,\cdot)$, that is in turn related to the stability term $\Stab(\cdot,\cdot)$. We therefore study the existence of a local positive constant $\alpha_E$ (for all elements $E$) such that
\begin{equation}\label{S:coerc}
\alpha_{E} \ a^E_h(v_h, v_h) \gtrsim a^E(v_h, v_h) 
\qquad \text{for all $v_h \in \Vem$.}
\end{equation}
Note that such condition immediately implies the corresponding global one 
by summing over all elements, with global constant
\begin{equation}\label{eq:alpha}
\alpha := \sup_{E \in \Omega_h} \alpha_{E} \,.
\end{equation}
It is immediate to check that both bilinear forms $\Stab(\cdot,  \cdot)$, cf. \eqref{eq:dofidofi} and \eqref{eq:trace}, are the restriction to $[\Vem + \Pkint(E)]$ of the classical corresponding discrete VEM forms on $V^h_{\kint}(E)$ 
(recalling  that $[\Vem + \Pkint(E)] \subseteq V^h_{\kint}(E)$).
Therefore the coercivity follows from existing results for standard VEM spaces. Since form \eqref{eq:trace} was shown in \cite{BLR:2017,brenner-sung:2018} to guarantee \eqref{S:coerc} with uniform constants, under the assumption $\textbf{(A1)}$ such stabilization yields bound \eqref{S:coerc} with constant $\alpha_{E}$ independent of any other geometric parameter.
\begin{lemma}
\label{lm:alphatrace}
Under assumption \textbf{(A1)}, for the choice \eqref{eq:trace} the bound \eqref{S:coerc} holds with constant $\alpha_{E} = 1$.
\end{lemma}

The results for the form \eqref{eq:dofidofi} are less favorable, since the results in the literature \cite{BLR:2017, brenner-sung:2018} assume an uniformly bounded number of edges (an assumption that would be unacceptable in the present study). A key role in our analysis is taken by the following lemma; the proof is quite technical and can be found in the Appendix.
\begin{lemma}\label{lem:hum} 
Let $\{ {\cal T}_h \}_h$ denote a family of piecewise quasi-uniform grids, see Definition \ref{pqu}, on intervals $\{ I^h \}_h$. Then it exists a constant $C=C(\overline{m},\overline{c},k)$ such that
$$
| v_h |_{1/2,I^h}^2 \leq C \log(1 + R_h) \sum_{e \in {\cal T}_h} \| v_h \|_{L^{\infty}(e)}^2
\quad \text{for all $v_h \in {\mathbb S}_k({\cal T}_h)$,}
$$
where ${\mathbb S}_k({\cal T}_h)$ denotes the space of continuous piecewise polynomial functions of degree $k$, and where $R_h$ denotes the ratio among the maximum and the minimum element length of ${\cal T}_h$.
\end{lemma}
We can now present the following result. 
\begin{lemma}
\label{lm:alphadofi} 
Under assumptions \textbf{(A1)} and \textbf{(A2)}, for the choice \eqref{eq:dofidofi} the bound \eqref{S:coerc} holds with constant $\alpha_{E} = \log(1 + \hh_E)$.
\end{lemma}
\begin{proof}
To avoid repetition of previously published material, we present the proof briefly, referring to results in the existing literature. Essentially, as introduced in \cite{BLR:2017}, the main step in proving the local coercivity \eqref{S:coerc} is showing that the boundary norm associated to $\Stab_{\texttt{d}}(\cdot,  \cdot)$ controls the $H^{1/2}(\partial E)$ seminorm for any function $v_h$ in the local VEM space $V^h_{\kint}(E)$. It is immediate to check that, for the choice \eqref{eq:dofidofi}, it holds
$$
\sum_{e \in {\cal T}_h} \| w_h \|_{L^{\infty}(e)}^2 \leq C \, \Stab_{\texttt{d}}(w_h,w_h)
\qquad \text{for all $w_h \in V^h_{\kint}(E)$,}
$$
with $C=C(\kbou)$. We now combine the above bound with Lemma \ref{lem:hum}, and apply it to the function $w_h - {\cal R} w_h$ with ${\cal R}$ the projection operator on $\Pk_0(E)$ given by the boundary average (${\cal R} w = \frac{1}{|\partial E|} \int_{\partial E} w$ for all $w \in H^1(E)$).
We obtain 
\begin{equation}\label{falcao}
|w_h |_{1/2,I^h}^2 = | w_h - {\cal R} w_h |_{1/2,I^h}^2 
\leq C \log(1 + \mathcal{H}_E) \, \Stab_{\texttt{d}}(w_h - {\cal R} w_h , w_h - {\cal R} w_h) \, ,
\end{equation}
for all $w_h \in V^h_{\kint}(E)$, which is exactly the boundary norm control mentioned above. 
Bound \eqref{falcao} allows to apply Proposition 3.6 in \cite{BLR:2017}, yielding (for all $w_h \in V^h_{\kint}(E)$)
$$  
a^E(w_h,w_h) \lesssim \log(1 + \mathcal{H}_E) \, \Stab_{\texttt{d}}(w_h - {\cal R} w_h , w_h - {\cal R} w_h) + |\PN w_h|_{1,E}^2  \, .
$$
By applying the above bound to  $w_h = v_h - \PN v_h$, for any $v_h \in V^h_{\kint}(E)$, we get
\begin{equation}\label{nakata} 
a^E(v_h - \PN v_h,v_h - \PN v_h) \lesssim \log(1 + \mathcal{H}_E) \, \Stab_{\texttt{d}}(v_h - \PN v_h, v_h - \PN v_h) \ .
\end{equation}
The result now follows immediately using a triangle inequality, bound \eqref{nakata} and definition \eqref{eq:ah-E-def}
$$
\begin{aligned}
a^E(v_h,v_h) &\leq a^E(\PN v_h,\PN v_h) + a^E(v_h - \PN v_h,v_h - \PN v_h) \\
& \lesssim \log(1 + \mathcal{H}_E) \, a^E_h (v_h,v_h) \, ,
\end{aligned}
$$
for all $v_h \in V^h_{\kint}(E)$.
\end{proof}


\subsection{Virtual element problem}
\label{sub:vemp}

Referring to the discrete space \eqref{eq:space-g}, the discrete bilinear form \eqref{eq:a_h}  and the approximated right-hand side \eqref{eq:right}, the virtual element approximation of the Poisson equation \eqref{eq:poisson-c} is 
\begin{equation}
\label{eq:poisson-vem}
\left \{
\begin{aligned}
& \text{find $u_h \in \VemG$ s.t.} 
\\
& a_h(u_h, \, v_h) = ( f_h, \, v_h) \qquad \text{for all $v_h \in \VemG$.}
\end{aligned}
\right.
\end{equation}
From \eqref{eq:alpha} it follows that Problem \eqref{eq:poisson-vem} has a unique solution $u_h \in \VemG$ such that 
$|u_h|_{1, \Omega} \lesssim \alpha \, \|f\|_{V^*}$.


\section{Convergence analysis}
\label{sec:convergence}

In this section we prove the interpolation estimates for the  virtual space $\VemG$ in \eqref{eq:space-g} and provide the error estimates for the solution of the discrete problem \eqref{eq:poisson-vem}.
All estimates are designed in order to distinguish the element interior and boundary contributions to the error, in terms of $h_E,h_{\partial E}, \kint,\kbou$.
We start by reviewing classical approximation result for polynomials on star-shaped domains, see for instance \cite{brenner-scott:book}.

\begin{lemma}[Bramble-Hilbert]
\label{lm:bramble}
Under the assumption \textbf{(A1)}, let two real non-negative numbers $r$, $s$ with $r \leq s \leq \kint + 1$. Then for all $v \in V \cap H^s(\Omega_h)$ there exists $v_{\pi} \in \Pk_{\kint}(\Omega_h)$ such that
 \[
 |v - v_{\pi}|_{\Omega_h,r} \lesssim \hbulk^{s - r} \, |v|_{\Omega_h,s} \,.
 \]
Moreover if $s>1$ then
 \[
 \|v - v_{\pi}\|_{L^{\infty}(\Omega)} \lesssim \hbulk^{s - 1} \, |v|_{\Omega_h,s} \,.
 \]
\end{lemma}

\subsection{Interpolation estimates}
\label{sub:interpolation}

%

In order to obtain clearer results, in the following proposition we assume ``maximum'' regularity of the target function (that is $v \in H^{\kint+1}(\Omega_h)$). 
Analogous results for $H^s(\Omega_h)$, $s \in (1,\kint+1)$, could be obtained by a more cumbersome argument involving space interpolation theory.

\begin{proposition}
\label{prp:interpolation}
Under the assumption \textbf{(A1)}, there exists a linear operator $\mathcal{I}_h \colon [V \cap H^s(\Omega_h)] \to \VemG$, with $s >1$, such that
$$
| v - \mathcal{I}_h v |_{1,E} \lesssim 
h_E^{\kint} \, |v|_{\kint+1,E} 
+  h_{\partial E}^{\kbou} \, |v|_{\kbou+1,E} \,, 
$$
for all $E \in \Omega_h$.
\end{proposition}
\begin{proof}
Let $v \in V \cap H^s(\Omega_h)$. 
On each element $E \in \Omega_h$ we consider the function 
$\mathcal{I}_h v$ defined by
\begin{equation}
\label{eq:Ih}
\left \{
\begin{aligned}
\Delta  \mathcal{I}_h v &= \Pi^{0,E}_{\kint -2}\Delta  v \quad & \text{in $E$,}
\\
\mathcal{I}_h v &= v_{b}  \quad & \text{on $\partial E$,}
\end{aligned}
\right.
\end{equation}
where $v_b$ is the standard 1D piecewise polynomial interpolation of $v|_{\partial E}$.
Therefore the interpolation error can be decomposed as
\begin{equation}
\label{eq:split}
v - \mathcal{I}_h v = \vint + \vp
\end{equation}
where
\begin{equation}
\label{eq:vv}
\left \{
\begin{aligned}
\Delta  \vint &= (I-\Pi^{0,E}_{\kint -2})\Delta  v \quad & \text{in $E$,}
\\
\vint &= 0  \quad & \text{on $\partial E$,}
\end{aligned}
\right.
\qquad \text{and} \qquad
\left \{
\begin{aligned}
\Delta  \vp &= 0 \quad & \text{in $E$,}
\\
\vp &= v - v_{b}  \quad & \text{on $\partial E$.}
\end{aligned}
\right.
\end{equation}
Notice that the splitting \eqref{eq:split} is $H^1$-orthogonal, i.e.
\begin{equation}
\label{eq:orth}
|v - \mathcal{I}_h v|_{1,E}^2 = |\vint|_{1,E}^2 + |\vp|_{1,E}^2 \,.
\end{equation}
For the first term, by equation \eqref{eq:vv}, classical stability results for the Poisson problem and Lemma \ref{lm:bramble}, we obtain
\begin{equation}
\label{eq:vint}
|\vint|_{1,E} = \|(I - \Pi^{0,E}_{\kint - 2})\Delta v\|_{-1,E} \lesssim h_E^{\kint} \, |\Delta v|_{\kint-1, E}
 \lesssim h_E^{\kint} \, | v|_{\kint+1, E} \,.
\end{equation}
Concerning the boundary term, again classical stability bounds  and standard polynomial interpolation results in one dimension yield
\begin{equation}\label{eq:L:1}
|\vp|_{1,E}^2 \lesssim |v - v_{b}|_{1/2, \partial E}^2 \lesssim h_{\partial E}^{2\kbou} 
\sum_{e \in \partial E} |v|^2_{\kbou+1/2,e} \, .
\end{equation}
It is immediate to check that, due to \textbf{(A1)}, for each edge $e \in \partial E$ it exists a triangle $T_e \subset E$ and all such triangles are disjoint and shape regular, uniformly in $E \in \Omega_h$ and $e \in \partial E$. Therefore if we apply a standard trace estimate on each of such triangles, from \eqref{eq:L:1} we obtain
$$
|v - v_{b}|_{1/2, \partial E}^2 
\lesssim h_{\partial E}^{2\kbou} \sum_{e \in \partial E} |v|^2_{\kbou+1,T_e}
\leq  h_{\partial E}^{2\kbou} |v|^2_{\kbou+1,E} \, .
$$
The above bound, combined with \eqref{eq:orth}, \eqref{eq:vint} and \eqref{eq:L:1} concludes the proof.
\end{proof}

Assuming additional (piecewise) regularity of the target function, another useful interpolation result can be obtained.

\begin{corollary}
\label{cor:interpolation}
Under the assumption \textbf{(A1)}, it exists a linear operator $\mathcal{I}_h \colon [V \cap H^s(\Omega_h)] \to \VemG$, with $s>1 $, such that
\[
|v - \mathcal{I}_h v |_{1,E} \lesssim 
h_E^{\kint} \, |v|_{\kint+1,E} 
+  h_{\partial E}^{\kbou+1/2} h_E^{-1/2} \, |v|_{\kbou+1,E}
+  h_{\partial E}^{\kbou+1/2} h_E^{1/2} \, |v|_{\kbou+2,E} \,, 
\]
for all $E \in \Omega_h$.
\end{corollary}
\begin{proof}
One follows the same steps as in the proof of Proposition \ref{prp:interpolation}, but the interpolation in \eqref{eq:L:1} is now stretched to its maximum reach in terms of polynomial approximation
\begin{equation}\label{eq:L:1bis}
|\vp|_{1,E}^2 \lesssim |v - v_{b}|_{1/2, \partial E}^2 \lesssim h_{\partial E}^{2\kbou + 1} 
\sum_{e \in \partial E} |v|^2_{\kbou+1,e} =  h_{\partial E}^{2\kbou + 1}  |v|_{\kbou+1, \partial E}^2\, .
\end{equation}
We then bound the $H^{\kbou+1}$ seminorm of $v$ on each edge $e$ by the $L^2(e)$ norm of the corresponding multi-index derivative matrix $D^{\kbou+1}v$ in 2D. Afterwards, by applying Lemma 6.4 in \cite{BLR:2017} we get
\begin{equation}
\label{eq:trace_s}
|v|_{\kbou+1, \partial E}^2  \lesssim h_E^{-1} \| D^{\kbou+1}v \|_{0,E}^2 
+ h_E | D^{\kbou+1}v|_{1,\partial E}^2 =  h_E^{-1} | v |_{\kbou+1,E}^2 
+ h_E | v |_{\kbou+2, E}^2  \, .
\end{equation} 
Therefore we obtain from \eqref{eq:L:1bis}
\[
|\vp|_{1,E}^2 
\lesssim h_{\partial E}^{2\kbou + 1} \big( h_E^{-1} | v |_{\kbou+1,E}^2 
+ h_E | v |_{\kbou+2, E}^2 \big) \, .
\]
The above bound, combined with  \eqref{eq:vint} and \eqref{eq:orth}  concludes the proof.
\end{proof}

\begin{remark} [$L^2$-interpolation estimate]
\label{rm:interpolation}
The Poincar\'e inequality and classical polynomial approximation result in 1D 
imply
\[
\begin{aligned}
\|v - \mathcal{I}_h v \|_{0,E} & \lesssim 
\int_{\partial E} |v - \mathcal{I}_h v| \, {\rm d}s + 
h_E |v - \mathcal{I}_h v |_{1,E} 
\\
& \lesssim 
h_{E}^{1/2} \|v - v_b\|_{0, \partial E} + 
h_E |v - \mathcal{I}_h v |_{1,E} 
\\
& \lesssim 
h_{E}^{1/2} h_{\partial E}^{\kbou+1} |v|_{\kbou+1, \partial E} + 
h_E |v - \mathcal{I}_h v |_{1,E}  \,.
\end{aligned}
\]
Therefore the bound above, bound \eqref{eq:trace_s} and Corollary \ref{cor:interpolation} entail the following $L^2$-interpolation estimate (recall also that $h_{\partial E} \le h_E$)
\begin{equation}
\label{eq:L2-int}
\|v - \mathcal{I}_h v\|_{0,E} \lesssim
h_E^{\kint+1} \, |v|_{\kint+1,E} 
+  h_{\partial E}^{\kbou+1/2} h_E^{1/2} \, |v|_{\kbou+1,E}
+  h_{\partial E}^{\kbou+1/2} h_E^{3/2} \, |v|_{\kbou+2,E} \,.
\end{equation}
\end{remark}

\begin{remark} [$L^\infty$-boundary estimate]
\label{rm:Linf-bound}
Combining standard one dimensional $L^\infty$ interpolation bounds with \eqref{eq:trace_s} yields
\begin{equation}\label{eq:Linf-bound}
\|v - \mathcal{I}_h v\|_{L^\infty(\partial E)}^2 \lesssim
h_{\partial E}^{2\kbou + 1}  |v|_{\kbou+1, \partial E}^2 \lesssim
h_{\partial E}^{2\kbou + 1} \big( h_E^{-1} | v |_{\kbou+1,E}^2 + h_E | v |_{\kbou+2, E}^2 \big)  \, .
\end{equation}
\end{remark}

\subsection{Error estimates}
\label{sub:error}

The aim of the present section is to derive the rate of convergence for the proposed virtual element scheme in terms of the mesh quantities $\hbulkE$, $\hpartialE$, $\hbulk$, $\hpartial$ and $\ell_E$, the coercivity constant $\alpha$ in \eqref{eq:alpha}, and the polynomial orders $\kint$ and $\kbou$.
We introduce the analysis with the following abstract error estimation.

\begin{proposition}
\label{prp:abstract}
Under the assumption \textbf{(A1)}, 
let $u \in V \cap H^s(\Omega_h)$ with $s>1$ be the solution of the equation \eqref{eq:poisson-c} and 
$u_h \in \VemG$ be the solution of the equation \eqref{eq:poisson-vem}. Consider the functions
\[
e_h := u_h - \mathcal{I}_h u \,, \qquad
e_{\mathcal{I}} := u - \mathcal{I}_h u \,, \qquad
e_{\pi} :=  u - u_{\pi} \,, \qquad
\emixed := u_{\pi} - \mathcal{I}_h u \,,
\]
where $\mathcal{I}_h u \in \VemG$ is the interpolant function of $u$ defined in \eqref{eq:Ih} and $u_{\pi} \in \Pk_{\kint}(\Omega_h)$ is the piecewise polynomial approximation of $u$ defined in Lemma \ref{lm:bramble}.
Then  it holds that
\begin{equation}
\label{eq:abstract}
| u - u_h |^2_{1, \Omega}  + \alpha \, a_h(e_h, \, e_h)\lesssim
\alpha^2 \!\! \sum_{E \in \Omega_h}
h_E^2 \, \|f - f_h\|^2_{0, E} + 
\alpha^2 \, |e_{\pi}|^2_{1, \Omega_h} +
\alpha \, |e_{\mathcal{I}}|^2_{1, \Omega} +
\alpha \!\! \sum_{E \in \Omega_h} \sigma^E 
\end{equation}
where $\alpha$ is the coercivity constant \eqref{eq:alpha} and 
$
\sigma^E := \mathcal{S}^E((I - \PN)e_u, (I - \PN)e_u) \,.
$
\end{proposition}
\begin{proof}
Simple computations yield
\begin{equation}
\label{eq:e_h}
\begin{aligned}
\alpha^{-1} &|e_h|^2_{1, \Omega} + a_h(e_h,  e_h)  \lesssim  a_h(e_h,  e_h) = a_h(u_h - \mathcal{I}_h u,  e_h) 
& \text{(by \eqref{S:coerc} and \eqref{eq:alpha})}
\\
& = (f_h - f,  e_h)  
+ a(u,  e_h) - a_h(\mathcal{I}_h u,  e_h)  
& \text{(using \eqref{eq:poisson-c} and \eqref{eq:poisson-vem})} 
\\
& = (f_h - f,  e_h)  + 
\sum_{E \in \Omega_h} a^E(e_{\pi},  e_h) + 
\sum_{E \in \Omega_h} a_h^E(e_u, e_h)  
& \text{(property \eqref{eq:a_h-cons})} 
\\ 
& =: \eta_f + \eta_{\pi} + \eta_{h} \,.
\end{aligned}
\end{equation}
Let us analyse each term in \eqref{eq:e_h}.
The first term, using \eqref{eq:right}, \eqref{eq:P0_k^E} and the Cauchy-Schwarz inequality
 can be bounded as follows
\begin{equation}
\label{eq:eta_f}
\begin{aligned}
\eta_f & = \sum_E (f_h - f, e_h)_{0,E} = 
\sum_{E \in \Omega_h} (f_h - f, e_h - \Pi_0^{0,E} e_h)_{0,E} 
\\
& \leq \sum_{E \in \Omega_h}  h_E \, \|f - f_h\|_{0, E} |e_h|_{1,E} 
\leq \alpha \!\! \sum_{E \in \Omega_h}
h_E^2 \, \|f - f_h\|^2_{0, E}
+ \frac{1}{4}  \alpha^{-1} \, |e_h|^2_{1,\Omega}    
\end{aligned}
\end{equation}
The Cauchy-Schwarz inequality  applied to the second term  in \eqref{eq:e_h} entails
\begin{equation}
\label{eq:eta_pi}
\eta_{\pi} = \sum_E a^E(e_{\pi}, e_h)_{0,E} \leq 
\sum_{E \in \Omega_h} |e_{\pi}|_{1,E} \, |e_{h}|_{1,E}   
\leq \alpha \, |e_{\pi}|_{1, \Omega_h}^2 + \frac{1}{4} \, \alpha^{-1} \, |e_h|^2_{1,\Omega}  \,.
\end{equation}
Finally for the last term  in \eqref{eq:e_h}, using the continuity of $\PN$ with respect to the $H^1$-seminorm, we have
\begin{equation}
\label{eq:eta_h}
\begin{split}
&\eta_{h}  = 
\sum_{E \in \Omega_h} a_h^E(e_u,  e_h) \leq
\sum_{E \in \Omega_h} a_h^E(e_u,  e_u)^{1/2} \, a_h^E(e_h, e_h)^{1/2}
\\
&\leq
\frac{1}{4} a_h(e_h,  e_h) + 
\sum_{E \in \Omega_h} a_h^E(e_u,  e_u) 
\\
& \leq  \frac{1}{4} a_h(e_h,  e_h) + 
\sum_{E \in \Omega_h} \left(a^E(\PN e_u,   \PN e_u) +
 \mathcal{S}^E ((I- \PN)e_u, (I - \PN) e_u) \right)
\\
& \leq \frac{1}{4} a_h(e_h,  e_h) + 
\sum_{E \in \Omega_h} |e_u|_{1,E}^2 +
\sum_{E \in \Omega_h} \sigma^E 
\\
& \leq \frac{1}{4} a_h(e_h,  e_h) + 
2 |e_{\pi}|_{1,\Omega_h}^2 + 2 |e_{\mathcal{I}}|_{1,\Omega}^2 +
\sum_{E \in \Omega_h} \sigma^E \,.
\end{split}
\end{equation}
Collecting \eqref{eq:eta_f}, \eqref{eq:eta_pi} and \eqref{eq:eta_h} in \eqref{eq:e_h} we obtain
\[
\alpha^{-1} |e_h|^2_{1, \Omega} + a_h(e_h, \, e_h)  \lesssim
\alpha \sum_{E \in \Omega_h} h_E^{2} \, \|f - f_h\|^2_{0, E}  +
\alpha |e_{\pi}|_{1,\Omega_h}^2 +  |e_{\mathcal{I}}|_{1,\Omega}^2 +
\sum_{E \in \Omega_h} \sigma^E \,.
\]
The proof now follows from the bound above and the triangular inequality.
\end{proof}

The next step in the analysis consists in estimating the term $\sigma^E$ in \eqref{eq:abstract} for the \texttt{dofi-dofi} and the \texttt{trace} stabilization
(that we denote respectively by $\sigma^E_{\texttt{d}}$ and $\sigma^E_{\texttt{t}}$). 

\begin{lemma}
\label{lm:dofi}
Consider the \texttt{dofi-dofi} stabilization $\Stab_{\texttt{d}}(\cdot, \cdot)$ in \eqref{eq:dofidofi}.
Then, under assumption {\bf (A1)} 
\[
\sigma^E_{\texttt{d}}  \lesssim 
\ell_E 
\left(\|e_{\pi}\|_{L^{\infty}(E)}^2
+ |e_{\pi}|^2_{1,E}  
+ \|e_{\mathcal{I}}\|_{L^{\infty}(\partial E)}^2 
+ h_E^{-2} \, \|e_{\mathcal{I}}\|_{0,E}^2 + |e_{\mathcal{I}}|^2_{1,E}
\right) \,.
\]
\end{lemma}
\begin{proof}
We preliminary observe that for all $v_h \in [\Vem + \Pk_{\kint}(E)]$, given $\zeta := (I - \PN)v_h$,
by definition of DoFs $\mathbf{D_V}$ and the Cauchy-Schwarz inequality we have
\[
\Stab_{\texttt{d}}(\zeta, \zeta)  \leq
\sum_{\texttt{nodes} \, x_i} |v_h(x_i) - \PN v_h(x_i)|^2 
+ \sum_{\texttt{moments}} \frac{1}{|E|^2} \|m_i\|_{0,E}^2 \|(I - \PN)v_h\|^2_{0,E}  \, ,
\]
where the first sum is for all the nodes $x_i \in\partial E$ associated to 
$\mathbf{D_V1}$ and $\mathbf{D_V2}$  with $\kbou = \kint$.
Being $\|m_i\|_{L^{\infty}(E)} \leq 1$ from the above inequality we infer
\[
\Stab_{\texttt{d}}(\zeta, \zeta)  \lesssim
\ell_E \left(\|v_h\|_{L^{\infty}(\partial E)}^2 + \|\PN v_h\|_{L^{\infty}(\partial E)}^2 \right) 
+ h_E^{-2} \|(I - \PN)v_h\|^2_{0,E} \,.
\]
From the bound above, recalling \eqref{eq:Pn_k^E} and using a scaled Poincar\'e inequality we get
\[
\Stab_{\texttt{d}}(\zeta, \zeta)
\lesssim
\ell_E \left(\|v_h\|_{L^{\infty}(\partial E)}^2 + \|\PN v_h
\|_{L^{\infty}(\partial E)}^2 \right) 
+ |(I - \PN)v_h|^2_{1,E}  \, .
\]
Furthermore, a standard scaling argument for polynomials and the continuity of $\PN$ with respect to the $H^1$ (scaled) norm entail the  estimate
\[
\Stab_{\texttt{d}}((I - \PN)v_h, (I - \PN)v_h)
\lesssim
\ell_E   \left(\|v_h\|_{L^{\infty}(\partial E)}^2 +  h_E^{-2} \, \|v_h\|_{0,E}^2 + |v_h|^2_{1,E} \right)  \,.
\]
Recalling that $e_u = e_{\mathcal{I}}- e_{\pi}$, we employ the bound above in order to estimate $\sigma^E_{\texttt{d}}$ obtaining
\[
\begin{aligned}
&\sigma^E_{\texttt{d}}   \lesssim
\ell_E \left( \|e_u\|_{L^{\infty}(\partial E)}^2 + h_E^{-2} \, \|e_u\|_{0,E}^2 + |e_u|^2_{1,E} \right)
\\
& \lesssim
\ell_E 
\left(\|e_{\pi}\|_{L^{\infty}(E)}^2 
+ h_E^{-2} \, \|e_{\pi}\|_{0,E}^2 
+ |e_{\pi}|^2_{1,E}  
+ \|e_{\mathcal{I}}\|_{L^{\infty}(\partial E)}^2 
+ h_E^{-2} \, \|e_{\mathcal{I}}\|_{0,E}^2 + |e_{\mathcal{I}}|^2_{1,E}
\right) \,.
\end{aligned}
\]
The result now follows from the above inequality and a trivial bound of the $L^2$ norm by the $L^\infty$ norm.
\end{proof}

\begin{lemma}
\label{lm:trace}
Consider the \texttt{trace} stabilization $\Stab_{\texttt{t}}(\cdot, \cdot)$ in \eqref{eq:trace}.
Then, under assumption {\bf (A1)} 
\[
\sigma^E_{\texttt{t}}  \lesssim 
h_E \, |e_{\pi}|^2_{1, \partial E} +  |e_{\pi}|_{1, E}^2 +
h_E \, |e_{\mathcal{I}}|^2_{1, \partial E}  + |e_{\mathcal{I}}|_{1, E}^2 \,.
\]
\end{lemma}
\begin{proof}
We start by observing that for all $v_h \in [\Vem + \Pk_{\kint}(E)]$ it holds
\[
\begin{aligned}
\Stab_{\texttt{t}}((I - \PN)v_h, (I - \PN)v_h) &\lesssim
\Stab_{\texttt{t}}(v_h, v_h) +
\Stab_{\texttt{t}}(\PN v_h, \PN v_h) 
\\
& = 
h_E \int_{\partial E} (\partial_s v_h)^2 \, {\rm d}s + 
h_E \int_{\partial E} (\partial_s \PN v_h)^2 \, {\rm d}s 
\\
& \le
h_E \, |v_h|^2_{1, \partial E} + h_E \|\nabla \PN v_h\|_{0, \partial E}^2
\\
& \lesssim
h_E \, |v_h|^2_{1, \partial E} + |v_h|_{1, E}^2
\end{aligned}
\]
where in the last inequality we first use a scaled trace inequality for polynomials and then the continuity of $\PN$ with respect to the $H^1$-seminorm.
Therefore the term $\sigma^E_{\texttt{t}}$ can be bounded as follows
\[
\sigma^E_{\texttt{t}} 
\lesssim h_E \, |e_u|^2_{1, \partial E} + |e_u|_{1, E}^2
\lesssim 
h_E \, |e_{\pi}|^2_{1, \partial E} + |e_{\pi}|_{1, E}^2 +
h_E \, |e_{\mathcal{I}}|^2_{1, \partial E}  + |e_{\mathcal{I}}|_{1, E}^2 \,. 
\]
\end{proof}

We are now ready to prove the following convergence results. For sake of simplicity, in accordance with Corollary \ref{cor:interpolation}, in both lemmas we assume 
all the needed (piecewise) regularity of the solution $u$. 

\begin{proposition}
\label{prp:errdofi}
Under the assumptions \textbf{(A1)} and  \textbf{(A2)},
let $u \in V$ be the solution of equation \eqref{eq:poisson-c} and 
$u_h \in \VemG$ be the solution of equation \eqref{eq:poisson-vem} obtained with the \texttt{dofi-dofi} stabilization (cf. \eqref{eq:dofidofi}).
Assume moreover that $u \in H^{\bar{k}}(\Omega_h)$  with $\bar{k} = \max\{\kint+1, \, \kbou+2 \}$ and $f \in H^{\kint-1}(\Omega_h)$. 
Then it holds that
\begin{equation}\label{eq:errordofi}
\begin{aligned}
|u - u_h|^2_{1, \Omega} & \lesssim
\alpha \sum_{E \in \Omega_h} 
\Big(
\alpha \, h_E^{2\kint} |f|^2_{\kint-1,E} + (\alpha + \ell_E) h_E^{2\kint} |u|^2_{\kint+1,E}  +
\\
& +  \ell_E \, h_{\partial E}^{2\kbou+1} h_E^{-1} \, |u|^2_{\kbou+1,E} 
+  \ell_E \, h_{\partial E}^{2\kbou+1} h_E \, |u|^2_{\kbou+2,E}
\Big)
\end{aligned}
\end{equation}
where $\alpha = \log(1 + \mathcal{H})$.
\end{proposition}
\begin{proof}
As direct consequence of Proposition \ref{prp:abstract} and Lemma \ref{lm:dofi} we get
\begin{multline}\label{eq:abstract-dofi}
|u - u_h|_{1, \Omega}^2  \lesssim 
\alpha^2 \!\! \sum_{E \in \Omega_h}  h_E^2 \, \|f - f_h\|^2_{0,E} +  \alpha^2 \, |e_{\pi}|^2_{1,\Omega_h} + 
\\
+ \alpha \! \sum_{E \in \Omega_h} \ell_E \left(\|e_{\pi}\|_{L^{\infty}(E)}^2 +  
|e_{\pi}|^2_{1,E} + 
\|e_{\mathcal{I}}\|_{L^{\infty}(\partial E)}^2 + 
h_E^{-2} \, \|e_{\mathcal{I}}\|_{0,E}^2 + |e_{\mathcal{I}}|^2_{1,E}
\right) .
\end{multline}
Recalling \eqref{eq:fh}, the Bramble-Hilbert Lemma \ref{lm:bramble} yields
\begin{equation}
\label{eq:err-bramble}
\begin{aligned}
& \sum_{E \in \Omega_h}  h_E^2 \, \|f - f_h\|^2_{0,E}  + |e_{\pi}|^2_{1, \Omega_h}
\lesssim \sum_{E \in \Omega_h} 
\big( h_E^{2\kint} |f|^2_{\kint-1,E} + h_E^{2\kint} |u|^2_{\kint+1,E} \big) \,,
\\
& \sum_{E \in \Omega_h} \ell_E \left(
\|e_{\pi}\|_{L^{\infty}(E)}^2 + |e_{\pi}|^2_{1,E} \right)
\lesssim \sum_{E \in \Omega_h} \ell_E \, h_E^{2\kint} |u|^2_{\kint+1, E}\,.
\end{aligned}
\end{equation}
Whereas from Corollary \ref{cor:interpolation}, \eqref{eq:L2-int} and \eqref{eq:Linf-bound}
we easily infer 
\begin{multline}\label{rossi}
\ell_E \left(
\|e_{\mathcal{I}}\|_{L^{\infty}(\partial E)}^2 + 
h_E^{-2} \, \|e_{\mathcal{I}}\|_{0,E}^2 + |e_{\mathcal{I}}|^2_{1,E} \right)
\lesssim 
\\
\lesssim
\ell_E \left(
h_E^{2\kint} \, |u|^2_{\kint+1,E} 
+  h_{\partial E}^{2\kbou+1} h_E^{-1} \, |u|^2_{\kbou+1,E}
+  h_{\partial E}^{2\kbou+1} h_E \, |u|^2_{\kbou+2,E} 
\right) \, .
\end{multline}
The proof follows taking the sum for all $E \in \Omega_h$ in the above bound, and combining it with \eqref{eq:abstract-dofi} and \eqref{eq:err-bramble}. Finally, the value $\alpha = \log(1 + \mathcal{H})$ follows from Lemma \ref{lm:alphadofi}.
\end{proof}

\begin{proposition}
\label{prp:errtrace}
Under the assumption \textbf{(A1)}, 
let $u \in V$ be the solution of the equation \eqref{eq:poisson-c} and 
$u_h \in \VemG$ be the solution of the equation \eqref{eq:poisson-vem} obtained with the \texttt{trace} stabilization (cf. \eqref{eq:trace}).
Assume moreover that $u \in H^{\bar{k}}(\Omega_h)$  with $\bar{k} = \max\{\kint+1, \, \kbou+2 \}$ and
$f \in H^{\kint-1}(\Omega_h)$, then it holds that
\begin{equation}
\label{eq:errortrace}
|u - u_h|^2_{1, \Omega} \lesssim
\sum_{E \in \partial E} \Big(
 h_E^{2\kint} |f|^2_{\kint-1,E} + 
 h_E^{2\kint} |u|^2_{\kint+1,E} + 
h_{\partial E}^{2\kbou} |u|^2_{\kbou+1,E} + h_{\partial E}^{2\kbou} h_E^2 |u|^2_{\kbou+2, E}  
\Big) \, .
\end{equation}
\end{proposition}
\begin{proof}
Proposition \ref{prp:abstract} and Lemma \ref{lm:dofi} combined with Lemma \ref{lm:alphatrace} imply
\begin{equation}
\label{eq:abstract-trace}
|u - u_h|_{1, \Omega}^2 \lesssim 
\sum_{E \in \Omega_h}  h_E^2 \, \|f - f_h\|^2_{0,E} +   |e_{\pi}|^2_{1,\Omega_h} + 
|e_{\mathcal{I}}|^2_{1,\Omega} 
+ \!\! \sum_{E \in \Omega_h}  \!\! \left( h_E |e_{\pi}|_{1, \partial E}^2 +   
h_E |e_{\mathcal{I}}|_{1, \partial E}^2 
\right) \,.
\end{equation}
Applying Lemma 6.4 in \cite{BLR:2017} and the Bramble-Hilbert Lemma \ref{lm:bramble} we obtain
\begin{equation}
\label{eq:err-tr1}
\sum_{E \in \Omega_h} h_E |e_{\pi}|_{1, \partial E}^2 \lesssim
\sum_{E \in \Omega_h} \left( |e_{\pi}|_{1, E}^2 + h_E^2 |e_{\pi}|_{2, E}^2 \right) 
\lesssim \sum_{E \in \Omega_h} h_E^{2\kint} |u|^2_{\kint+1,E} \,.
\end{equation}
Whereas polynomial approximation in 1D and bound \eqref{eq:trace_s} imply
\[
|e_{\mathcal{I}}|_{1, \partial E}^2 =
|u - \mathcal{I}_h u|_{1, \partial E}^2 \lesssim
h_{\partial E}^{2\kbou} |u|^2_{\kbou+1, \partial E} \lesssim
h_{\partial E}^{2\kbou} \left(h_E^{-1} |u|^2_{\kbou+1, E} + h_E|u|^2_{\kbou+2, E} \right)
\]
therefore
\begin{equation}
\label{eq:err-tr2}
\sum_{E \in \Omega_h} h_E |e_{\mathcal{I}}|_{1, \partial E}^2 \lesssim
\sum_{E \in \Omega_h} h_{\partial E}^{2\kbou} \left(|u|^2_{\kbou+1, E} + h_E^2 |u|^2_{\kbou+2, E} \right) \,.
\end{equation}
The thesis now follows gathering \eqref{eq:err-bramble}, \eqref{eq:err-tr1}, \eqref{eq:err-tr2} and Corollary \ref{cor:interpolation} in \eqref{eq:abstract-trace}, where we also make use of the trivial bound $h_{\partial E} \leq h_E$ to eliminate some terms.
\end{proof}

The error estimates in Proposition \ref{prp:errdofi} and Proposition \ref{prp:errtrace} separate the influence of the internal and boundary part of the elements, and are explicit in the parameters of interest. A simplified point of view, that helps understanding the implications of the above results, can be trivially derived including the Sobolev regularity terms (for $f$ and $u$) in the constant, assuming the reasonable relation $\ell_E \sim h_E/h_{\partial E}$ 
(that holds, for instance, for any quasi-uniform edge subdivision) and finally dropping the higher order terms. One obtains the estimates
\begin{equation}\label{paragone}
\begin{aligned}
|u - u_h|_{1, \Omega}^2  & \lesssim \alpha \sum_{E \in \Omega_h} 
\Big( (\alpha + \ell_E)^{1/2} \, h_E^{\kint} + h_{\partial E}^{\kbou} \Big)^2 
\qquad & \texttt{dofi-dofi} \, , \\
|u - u_h|_{1, \Omega}^2 & \lesssim \sum_{E \in \Omega_h} 
\Big( h_E^{\kint} + h_{\partial E}^{\kbou} \Big)^2 \qquad & \texttt{trace} \, .
\end{aligned}
\end{equation}
We draw some observation:
\\
$\bullet$ We recover the optimal rate of convergence in terms of  $h$ and $h_{\partial}$ that is $h^{\kint} + h_{\partial}^{\kbou}$. Therefore, if $\kint > \kbou$, the second term is expected to dominate; thus having smaller edges potentially leads to a more accurate solution.
\\
$\bullet$ The error estimates obtained with the \texttt{trace} stabilization is independent of $\mathcal{H}$ and $\ell_E$, thus are completely robust to any kind of edge refinement. 
\\
$\bullet$ For the \texttt{dofi-dofi} stabilization the error is polluted by   $\alpha = \log(1 + \mathcal{H})$ and $\ell_E^{1/2}$. 
The term $\log(1 + \mathcal{H})$ arises also in the analysis carried out in the papers \cite{BLR:2017,brenner-sung:2018} and is related to the presence of ``small edges''. Being  a logarithmic term, the influence is anyway minimal. 
Concerning  the dependence on $\ell_E^{1/2}$ we stress that such factor appears in front of the  ``higher'' order term $h^{\kint}$ (we recall that in our setting $\kint \geq \kbou$) therefore the influence of the number of edges $\ell_E$ is reduced. For $\kint > \kbou$ many small edges will in general lead to a more accurate solution, up to a certain extent.

\section{Numerical tests}
\label{sec:tests}

In this section we present some numerical experiments to be compared with our theoretical findings, also in order to test the practical aspects of increasing the internal degree $\kint$.
In Test 1 we examine the convergence properties of the proposed family of generalized VEM in the light of Proposition \ref{prp:errdofi} and Proposition \ref{prp:errtrace}.
In Test 2 and Test 3 we assess the behaviour of generalized VEM for a family of Voronoi meshes and a family of meshes arising from an agglomeration procedure.
In order to compute the VEM errors between the exact solution $u_{\rm ex}$ and the VEM solution $u_h$, we consider the computable $H^1$-like error quantities:
\begin{align}
\label{eq:err_bulk}
\texttt{err(bulk)}^2 &:=
\frac{\sum_{E \in \Omega_h} \|\nabla u_{\rm ex} - \Pi^{0,E}_{\kint-1} \nabla u_h  \|^2_{0,E}}
{| u_{\rm ex}|^2_{1\Omega}} \, ,
\\
\label{eq:err_trace}
\texttt{err(trace)}^2 &:=
\frac{\sum_{\texttt{edges} \, e}  H_e \int_{e} (\partial_s u_{\rm ex} - \partial_s u_h)^2 \, \rm ds}
{\sum_{\texttt{edges} \, e}  H_e \int_{e} (\partial_s u_{\rm ex})^2 \, \rm ds} \, ,
\end{align}
where $H_e$ denotes the average of the diameters of all the elements sharing the edge $e$. 
The error \texttt{err(bulk)} is the standard way to 
evaluate the $H^1$-seminorm VEM error. 
The error \texttt{err(trace)} also mimics a kind of $H^1$ discrete norm and involves the explicit value of the discrete solution on the skeleton of the mesh.

In the numerical tests we use the \texttt{dofi-dofi} stabilization \eqref{eq:dofidofi} and the \texttt{trace} stabilization \eqref{eq:trace}. 
For the dofi-dofi stabilization similar results are obtained with other variants such as the \texttt{D-recipe} stabilization introduced in \cite{BDR:2017}  or when adopting a lighter \texttt{dofi-dofi} stabilization in which the boundary evaluations are reduced from $\ell_E \, \kint$ to $\ell_E \, \kbou$, i.e. the \texttt{dofi-dofi} stabilization based on the true DoFs.

For both numerical tests we consider the Poisson equation on the unit square $\Omega = [0,1]^2$ and 
we choose the load term $f$ and the (non-homogeneous Dirichlet) boundary conditions in accordance with the analytical solution
\[
u_{\rm ex}(x, y) =  
x^5 + x^4y - xy^4  + x^3 - xy - x + y - 1
+ \sin(2\pi x)\sin(\pi y)
+ \log(x^2 + y^4 + 1) \,.
\]


\textbf{Test 1 (Convergence analysis)}
\label{test1}
The aim of the present test is to confirm the theoretical predictions of Proposition \ref{prp:errdofi} and Proposition \ref{prp:errtrace} and in particular the effective decoupling of the error into bulk and boundary components.
The domain is partitioned with two sequences of polygonal meshes: 
the uniform quadrilateral meshes $\mathcal{Q}_h$ and
the Voronoi meshes $\mathcal{V}_h$ (see Fig. \ref{fig:meshes_a}) with diameter $h=2^{-2}, \, 2^{-3}, \,2^{-4}, \, 2^{-5}$. 
For the generation of the Voronoi meshes we used the code Polymesher \cite{Polymesher}.
We then generate the sequences of meshes with uniform edge subdivision 
\begin{itemize}
\item $\mathcal{Q}_h^{\hpartial}$ with $\hpartial=2^{-1}h, \, 2^{-2}h, \,2^{-3}h, \, 2^{-4}h$ (see Fig. \ref{fig:meshes_b});
\item $\mathcal{V}_h^{\hpartial}$ with $\hpartial\approx h, \, 2^{-1}h, \,2^{-2}h, \, 2^{-3}h$ (see Fig. \ref{fig:meshes_c}).
\end{itemize}
Note that, since the Voronoi meshes have naturally smaller edges than square meshes (in comparison with the respective element diameter), the subdivisions above have a different range for the exponent in order to make the two cases comparable.
Furthermore, we observe that for the families of meshes above $\mathcal{H} \lesssim 1$, so that in accordance with Lemma \ref{lm:alphatrace} and Lemma \ref{lm:alphadofi} the coercivity constant in \eqref{eq:alpha} is $\alpha\lesssim 1$.
\begin{figure}[!htb]
\begin{center}
\begin{subfigure}{1\textwidth}
\begin{center}
\includegraphics[scale=0.18]{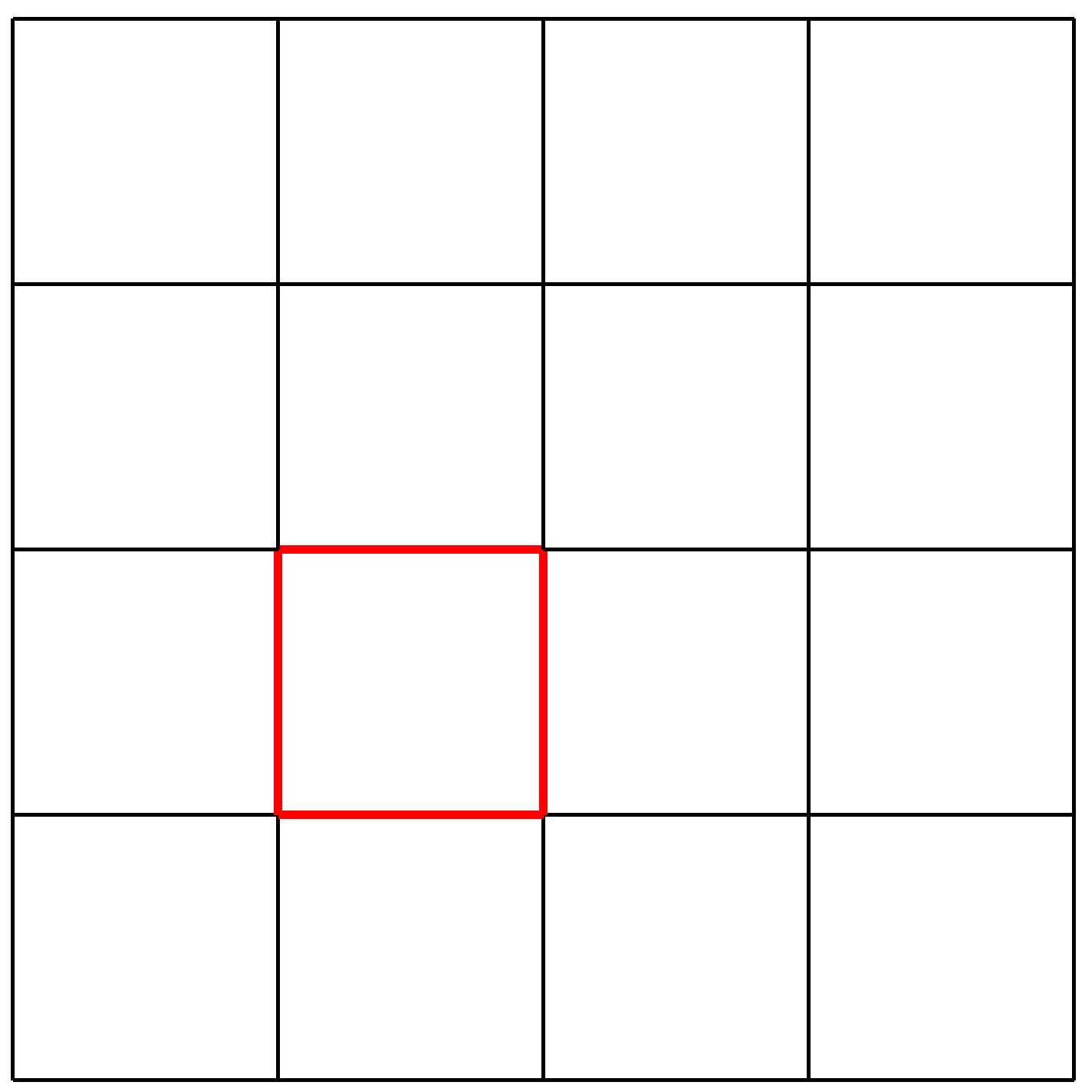}
\qquad \qquad
\includegraphics[scale=0.18]{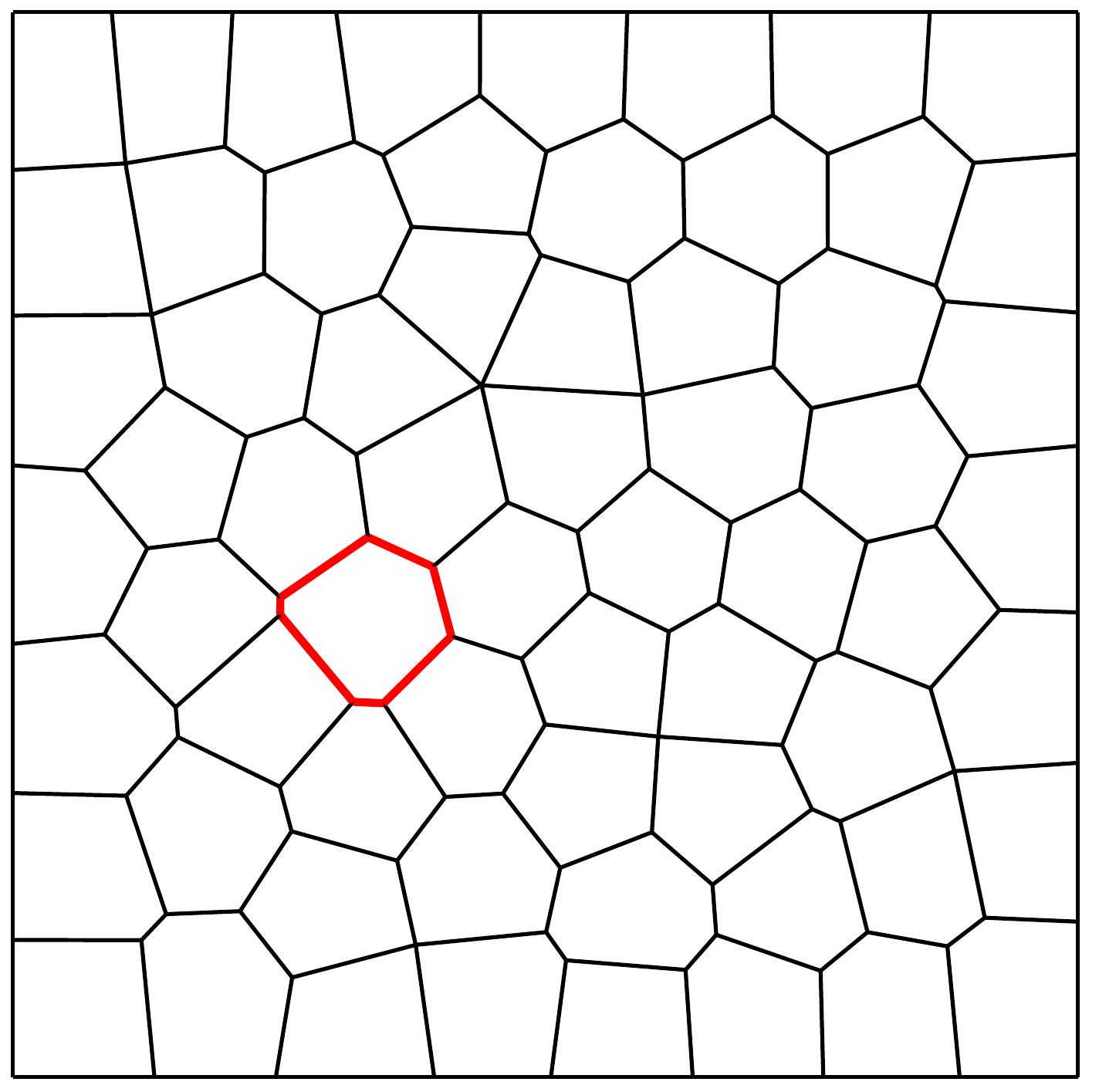}
\caption{Example of the adopted polygonal meshes: \texttt{QUADS} (left), \texttt{VORONOI} (right).}
\label{fig:meshes_a}
\end{center}
\end{subfigure}
\\
\begin{subfigure}{1\textwidth}
\begin{center}
\includegraphics[scale=0.18]{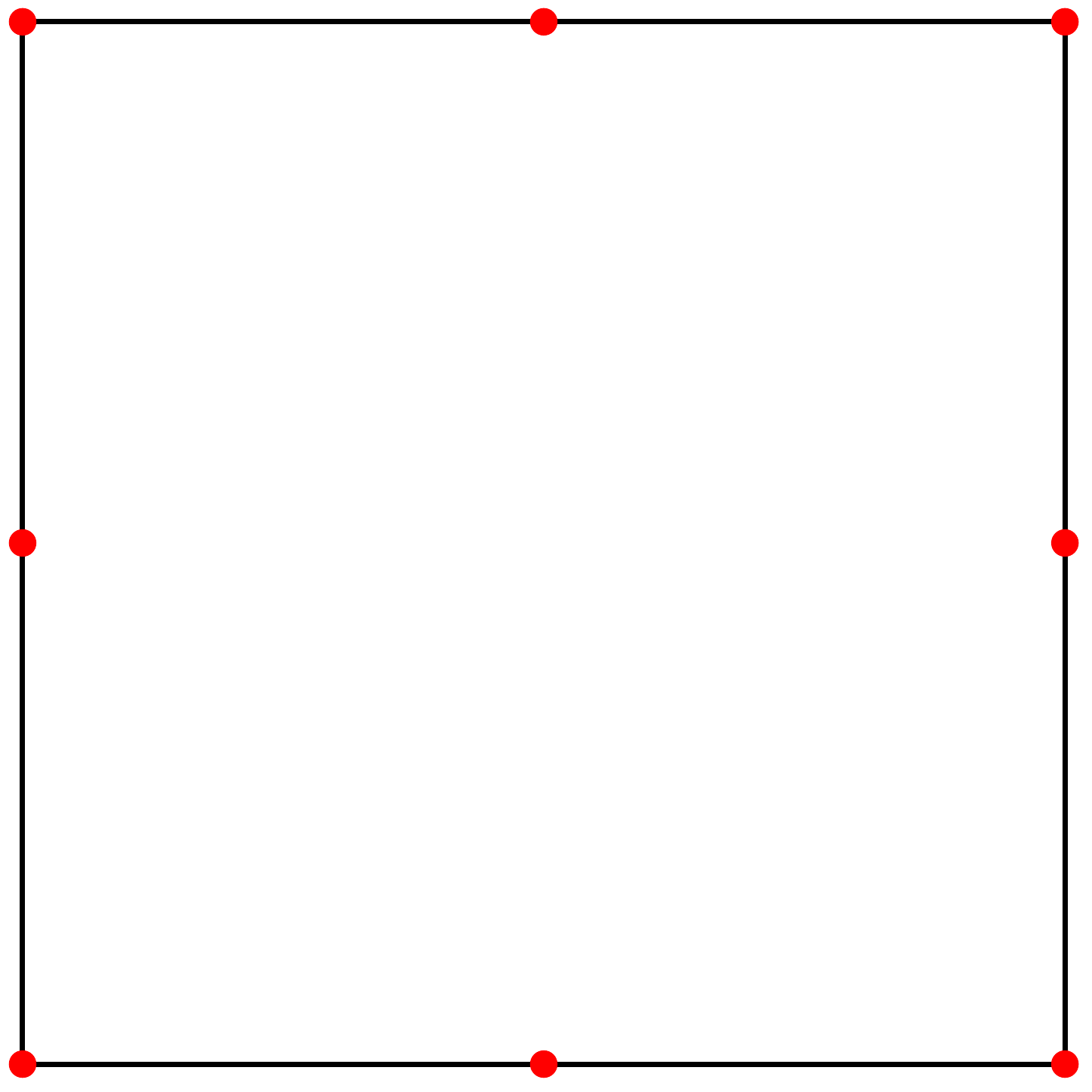}
\quad
\includegraphics[scale=0.18]{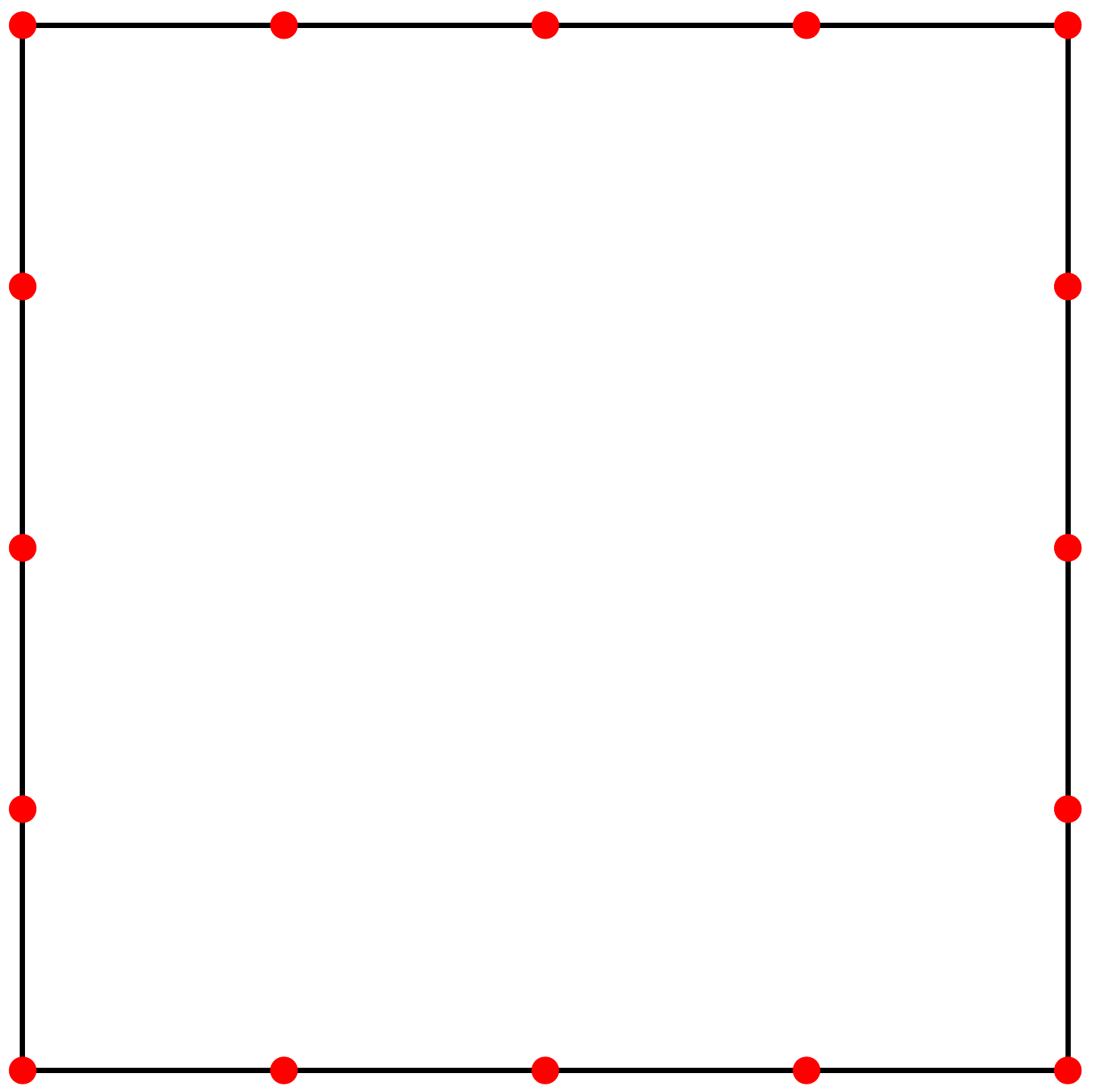}
\quad 
\includegraphics[scale=0.18]{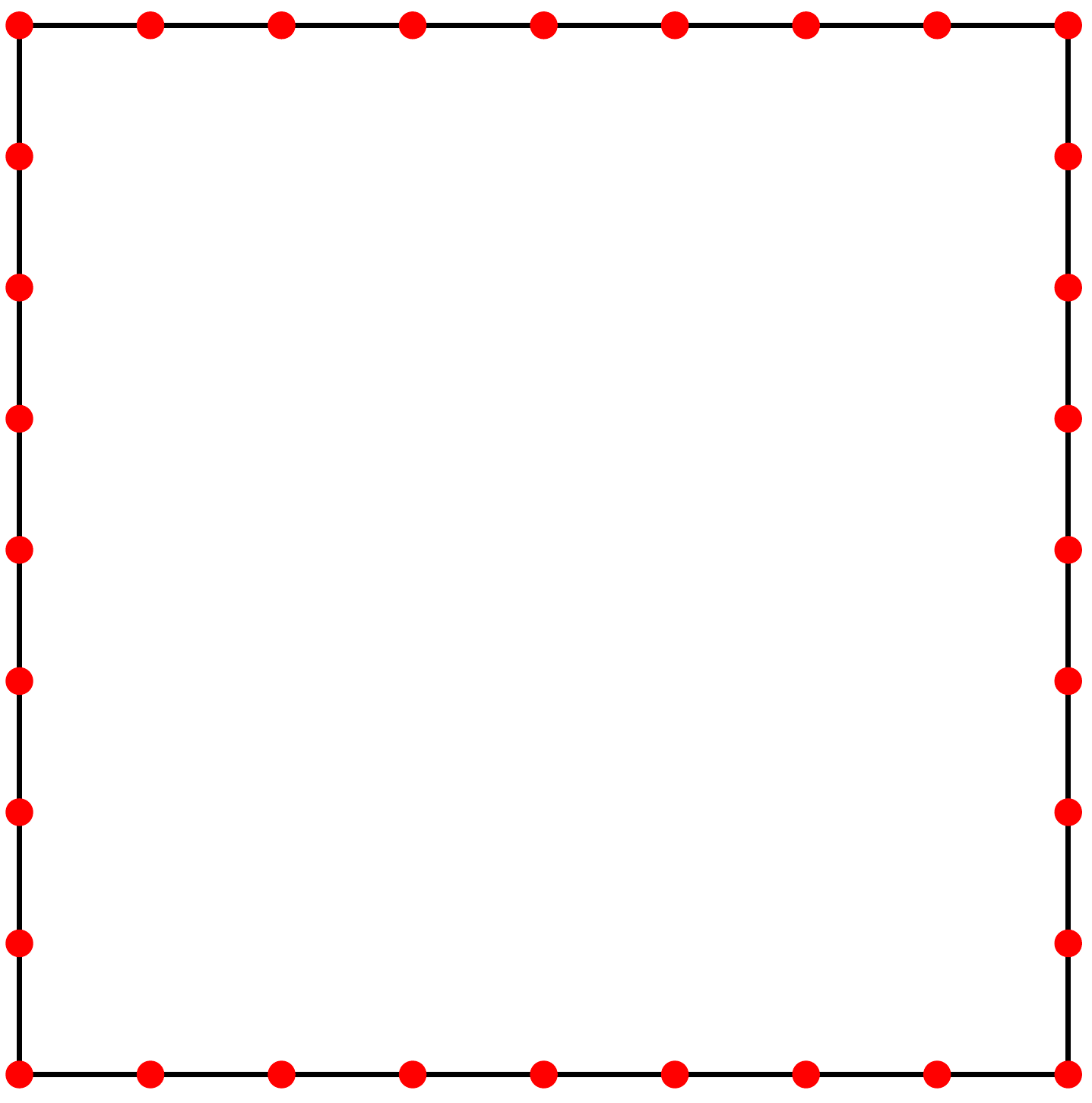}
\quad 
\includegraphics[scale=0.18]{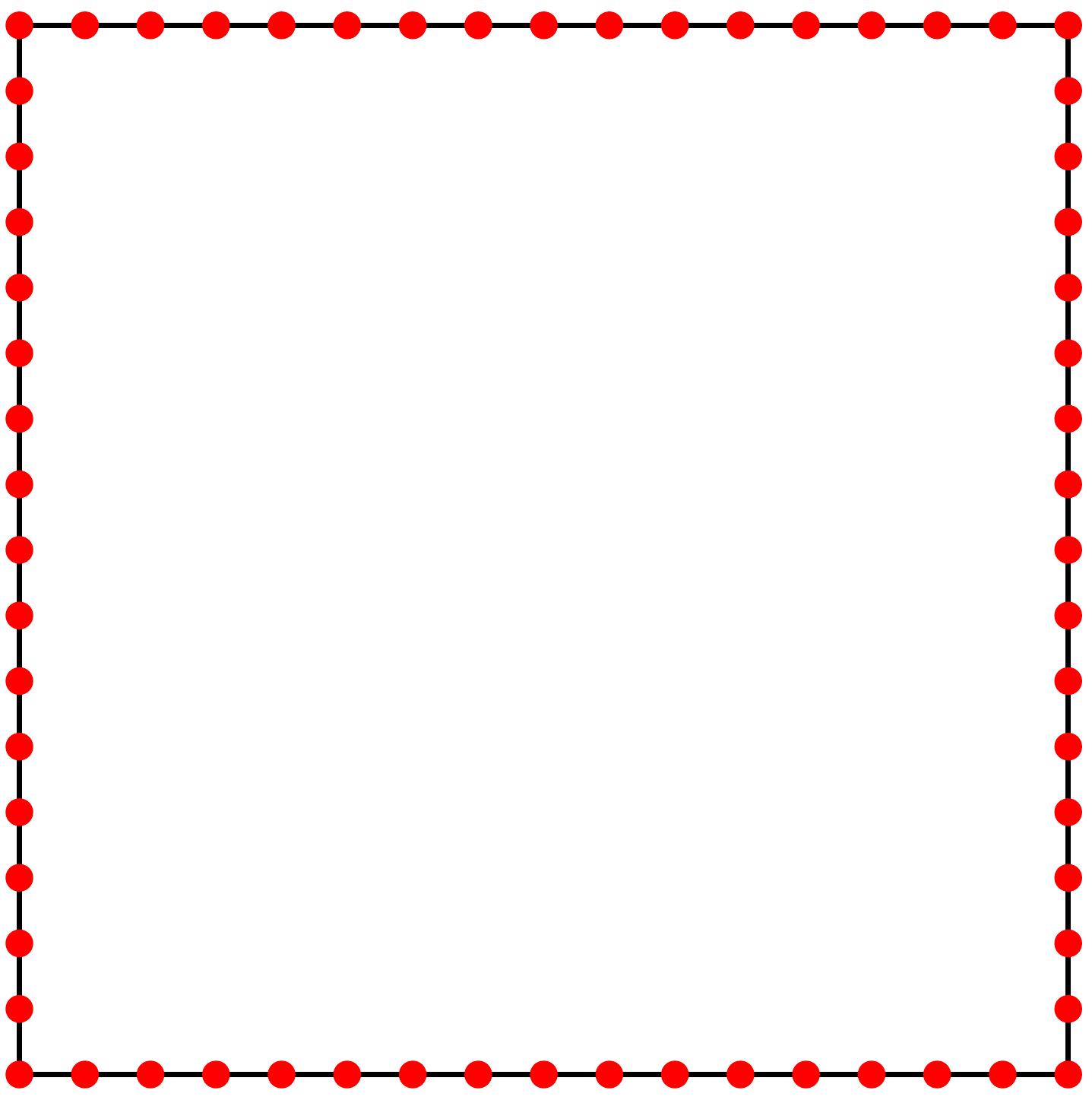}
\caption{Mesh element of the meshes $\mathcal{Q}_h^{\hpartial}$.}
\label{fig:meshes_b}
\end{center}
\end{subfigure}
\\
\begin{subfigure}{1\textwidth}
\begin{center}
\includegraphics[scale=0.2]{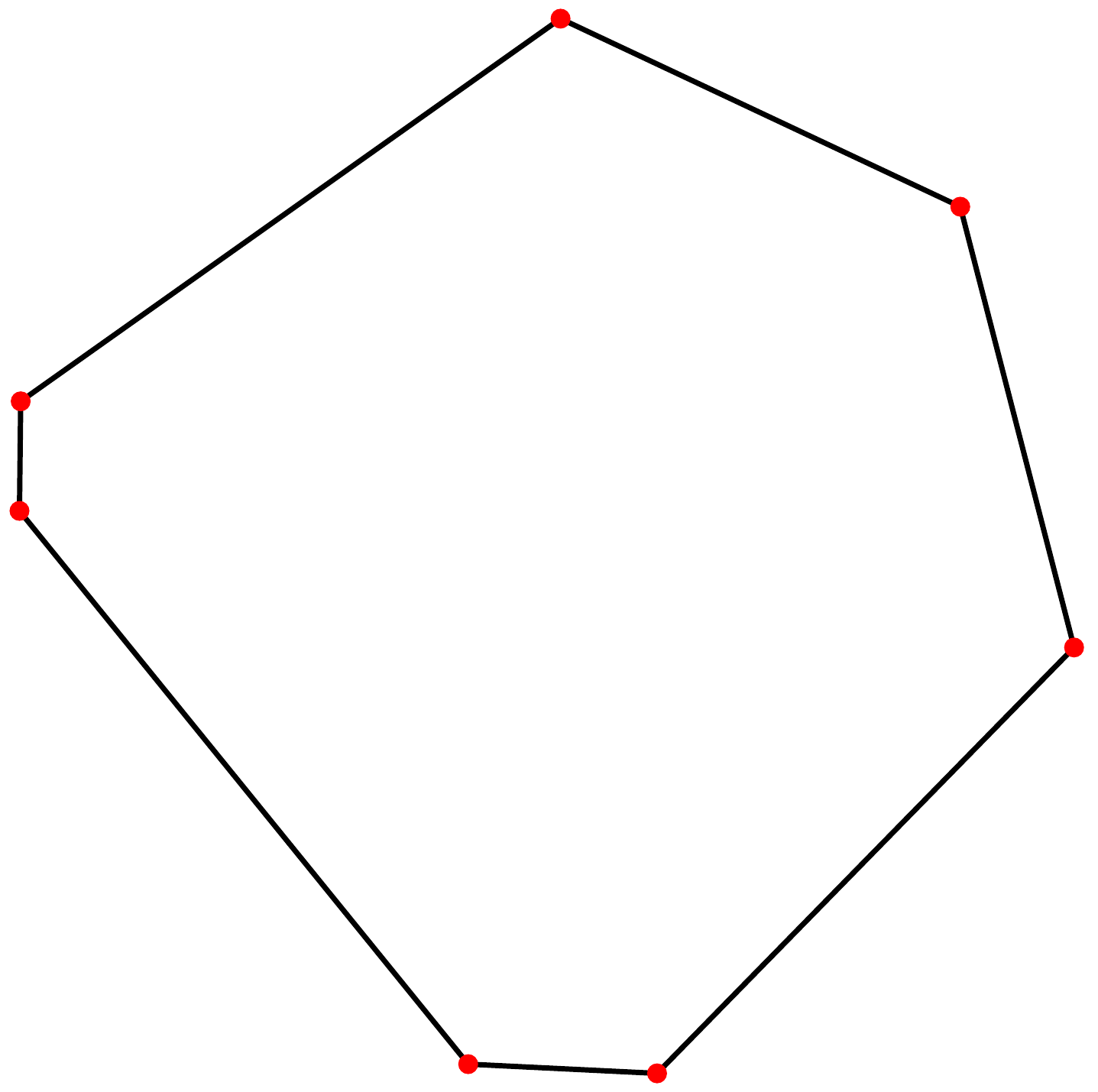}
\quad
\includegraphics[scale=0.2]{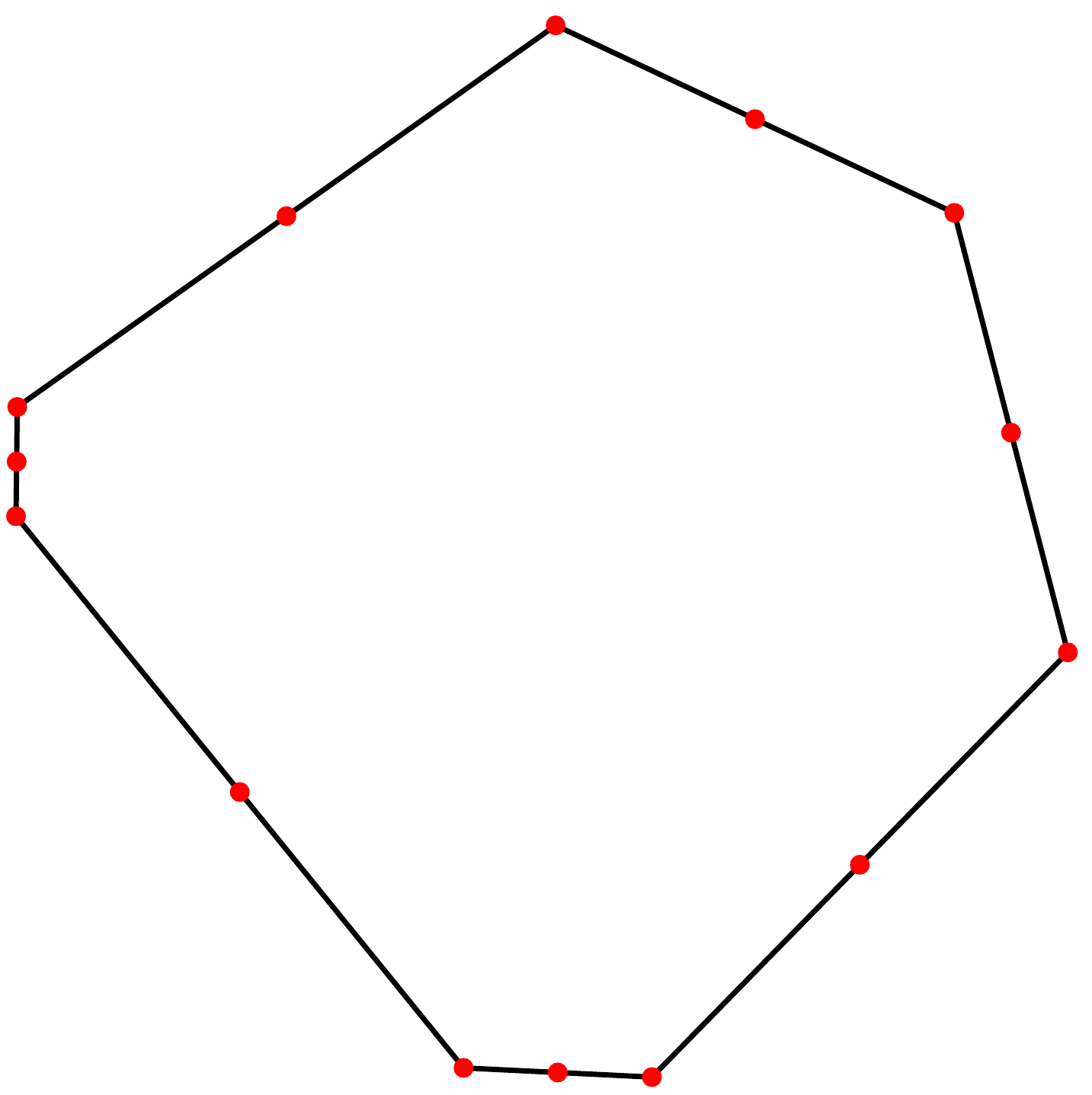}
\quad 
\includegraphics[scale=0.2]{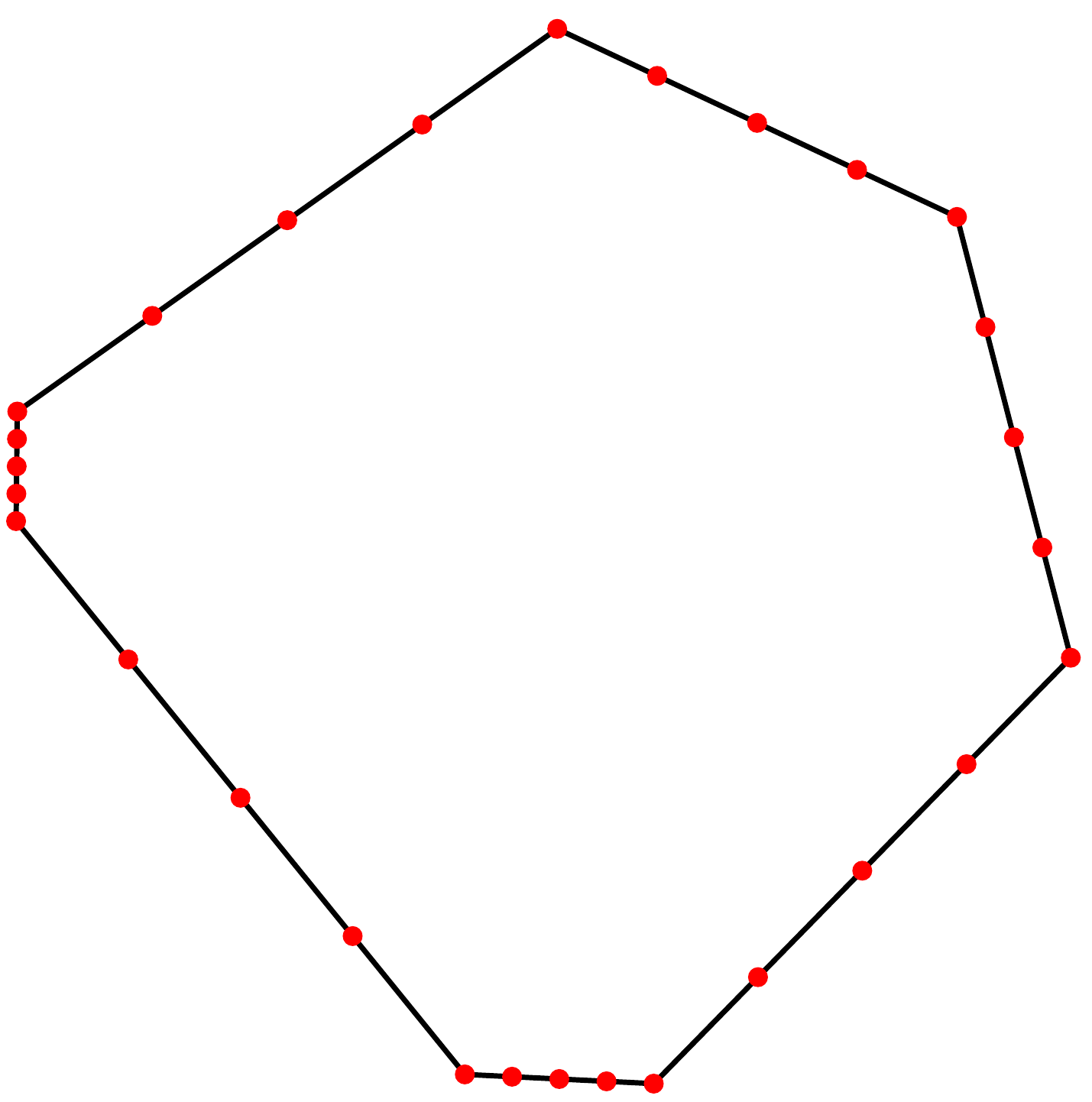}
\quad 
\includegraphics[scale=0.2]{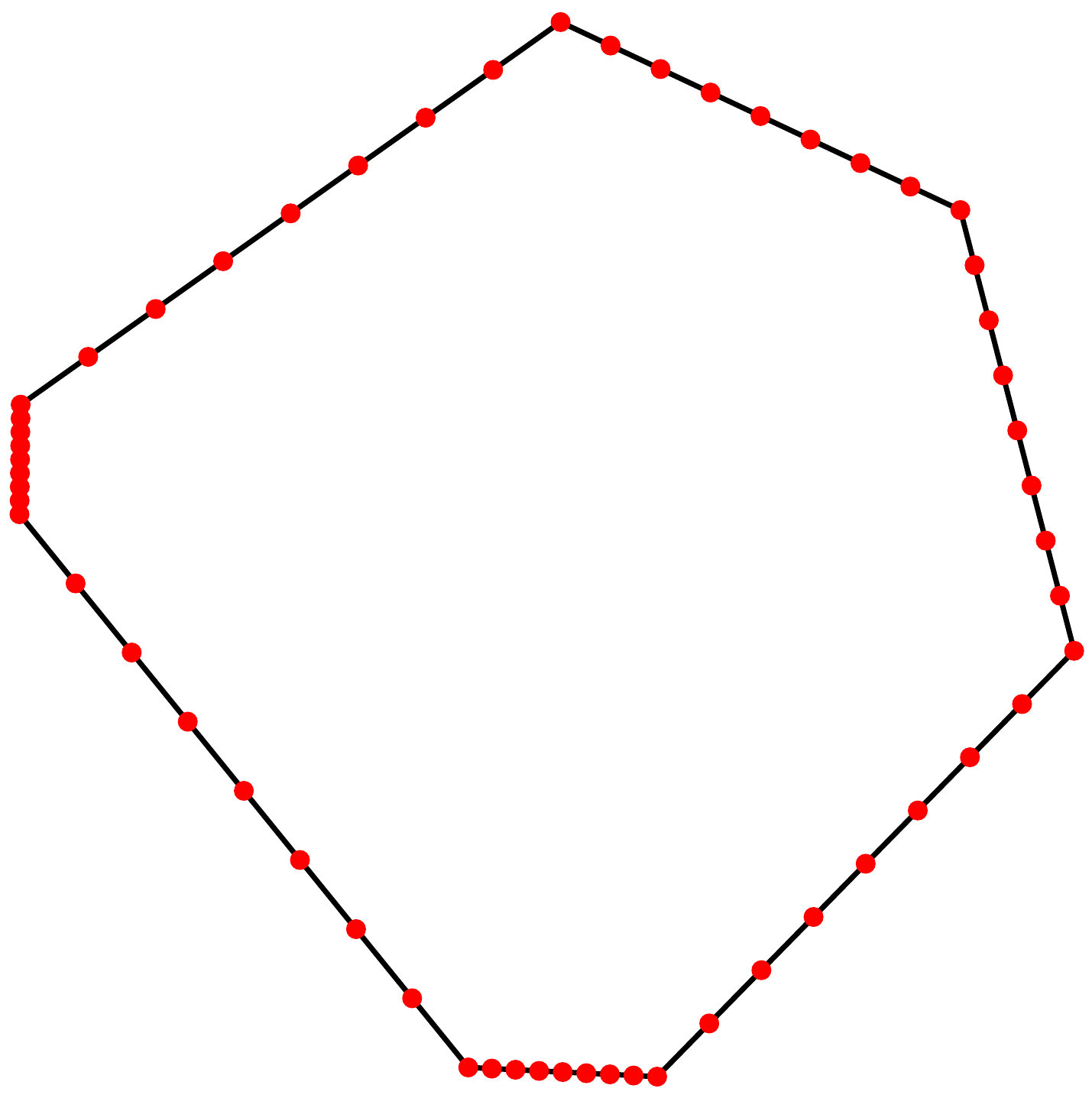}
\caption{Mesh element of the meshes $\mathcal{V}_h^{\hpartial}$.}
\label{fig:meshes_c}
\end{center}
\end{subfigure}
\end{center}
\caption{Test 1. Example of the adopted polygonal meshes and mesh elements.}
\label{fig:meshes}
\end{figure}

In Fig. \ref{fig:test1_1} we display the error \texttt{err(bulk)} and the error \texttt{err(trace)}
for the sequence of quadrilateral meshes $\mathcal{Q}_h^{2^{-4}h}$  i.e. the meshes with the finest edge refinement (the rightmost in Fig. \ref{fig:meshes_b}). 
For the meshes under consideration, $\hpartial$ is much smaller than $h$, therefore, in the light of \eqref{paragone}, we expect that the boundary component of the error is marginal with respect to the bulk component (at least for the considered ranges of $h$).
This phenomena is evident for the error \texttt{err(bulk)} where,  for both stabilizations, we recover the  order of convergence $O(h^{\kint})$, in full accordance with \eqref{paragone}. 
On the other hand, the error \texttt{err(trace)} is by nature direcly related to the mesh element boundary, and therefore one expects a stronger influence of the boundary component of the error. Nevertheless, we can still appreciate that the error is behaving essentially as $O(h^{\kint})$, apart in the more unbalanced case $\kk = (3,1)$, where one can still see the influence of the boundary component of the error.
Analogous results where obtained for the corresponding sequence of Voronoi meshes  $\mathcal{V}_h^{2^{-3}h}$ (not reported).
\begin{figure}[!htb]
\begin{center}
\includegraphics[width=0.49\textwidth]{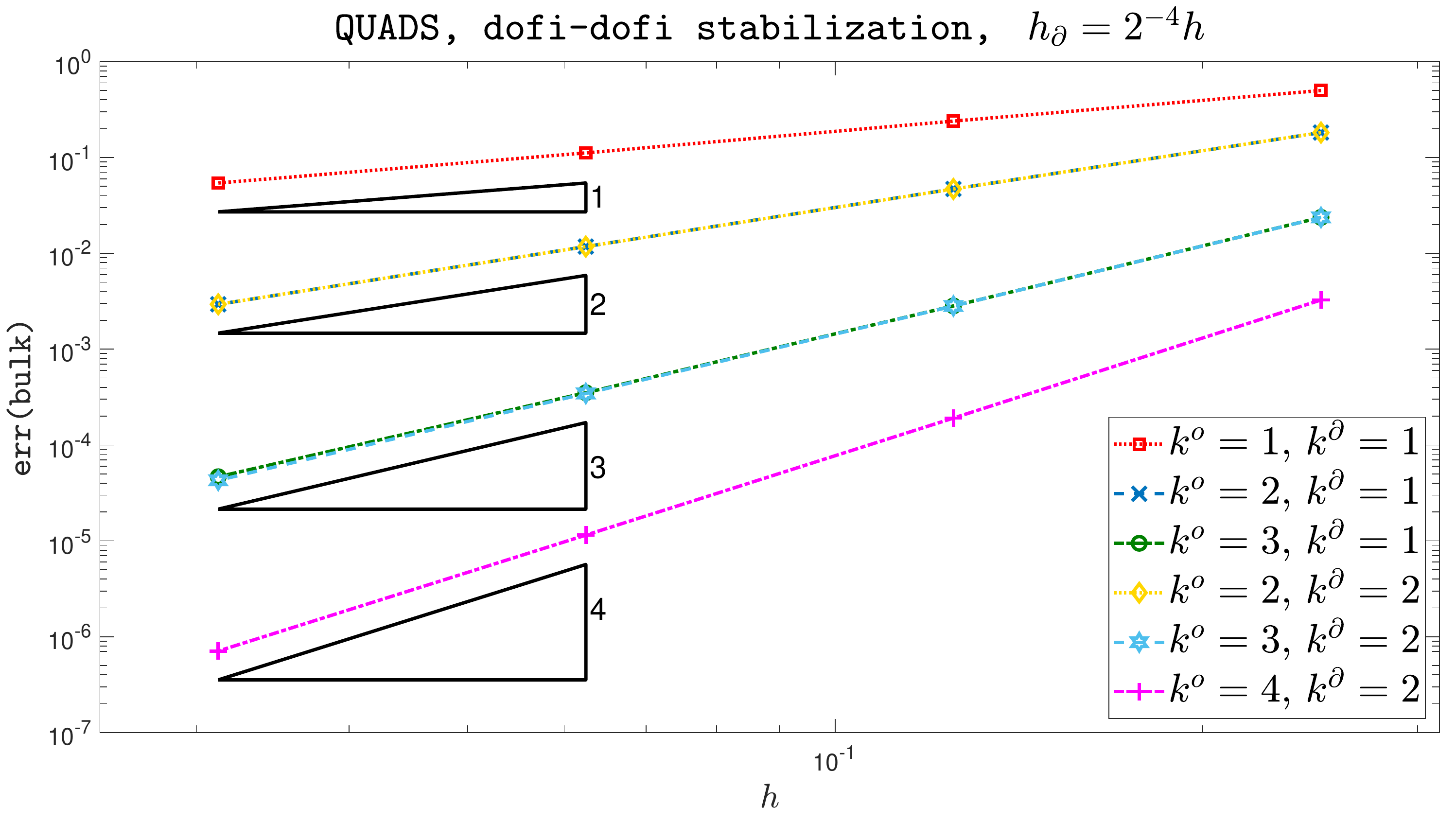}
\includegraphics[width=0.49\textwidth]{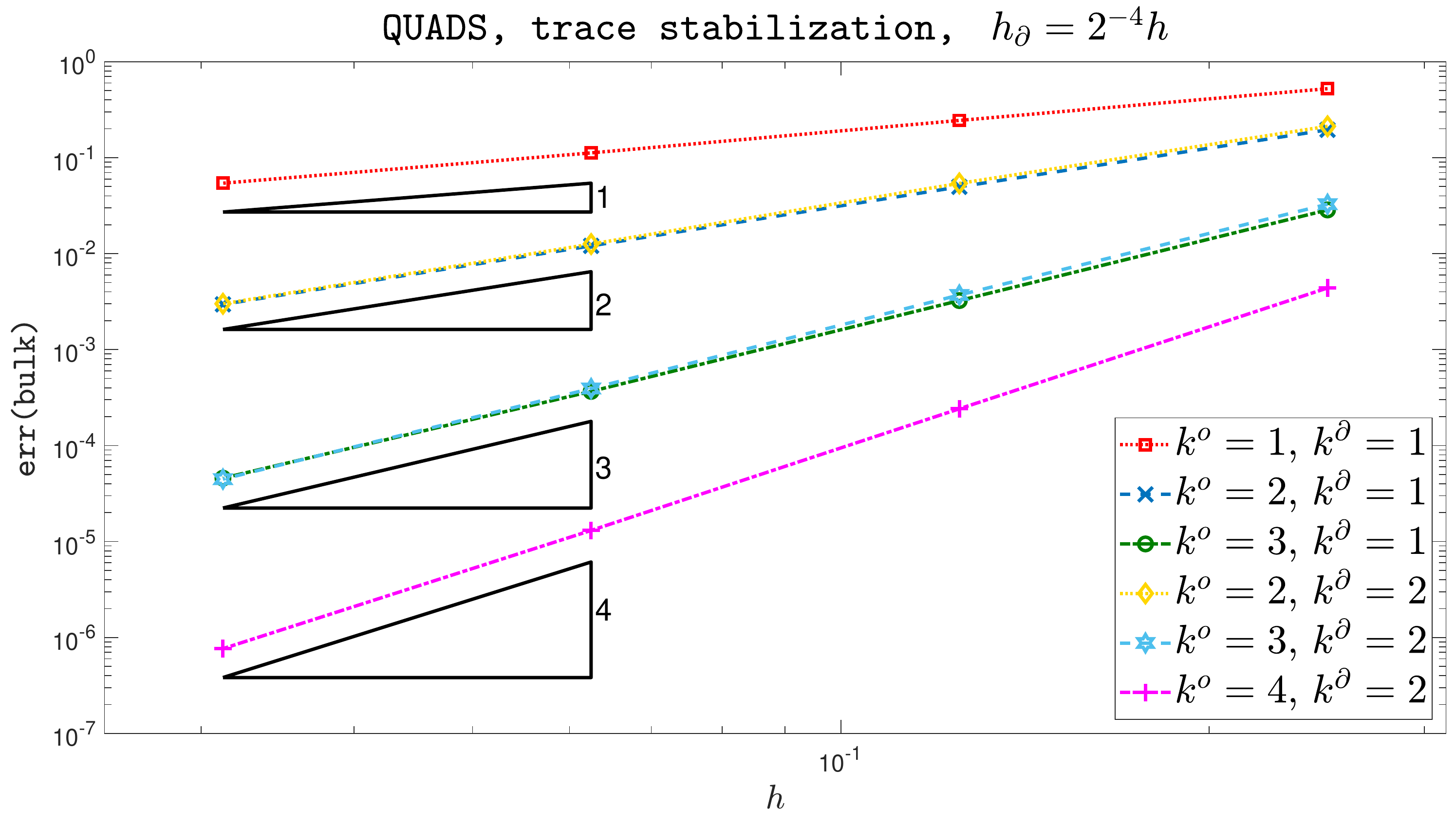}
\\
\includegraphics[width=0.49\textwidth]{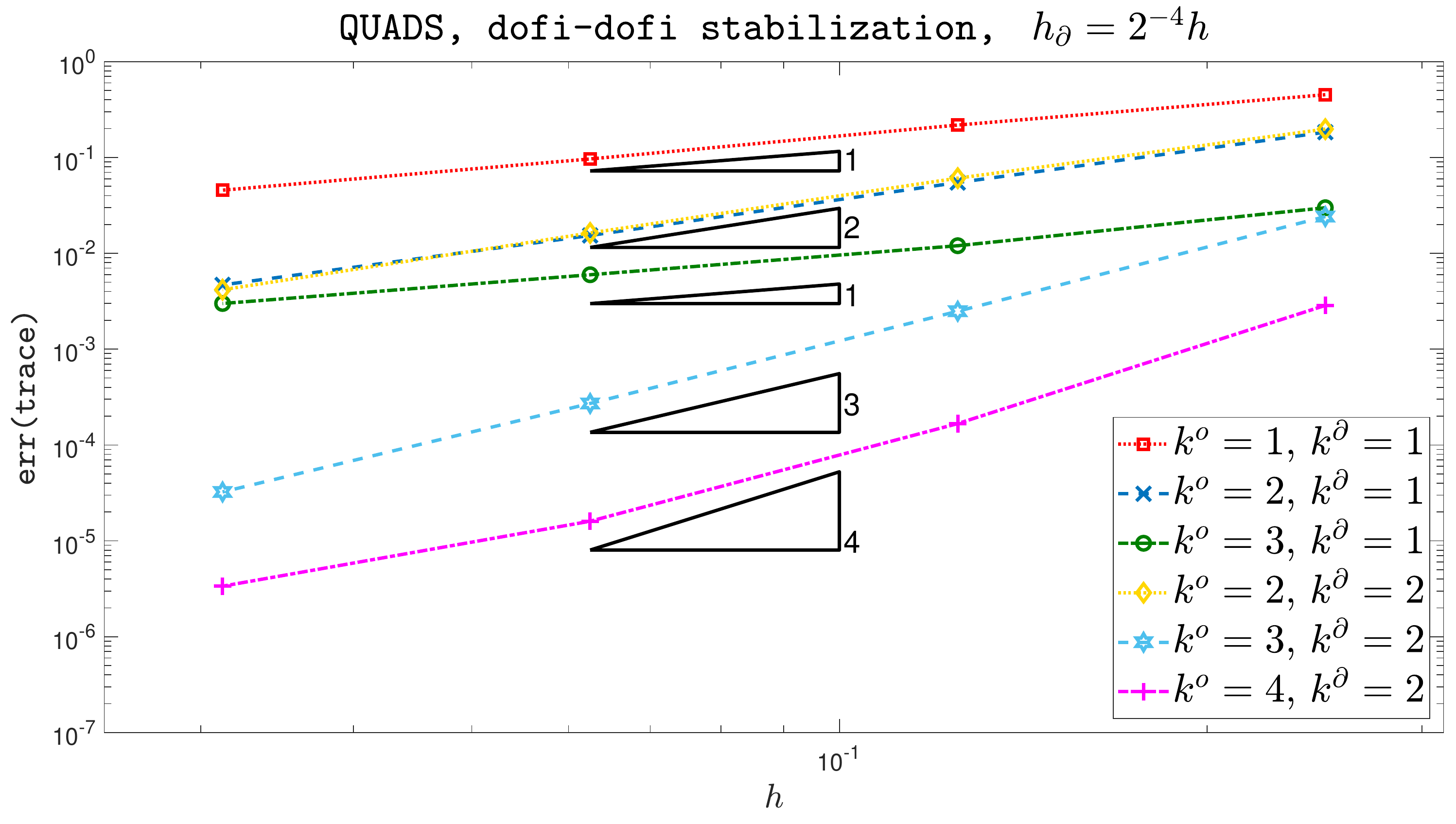}
\includegraphics[width=0.49\textwidth]{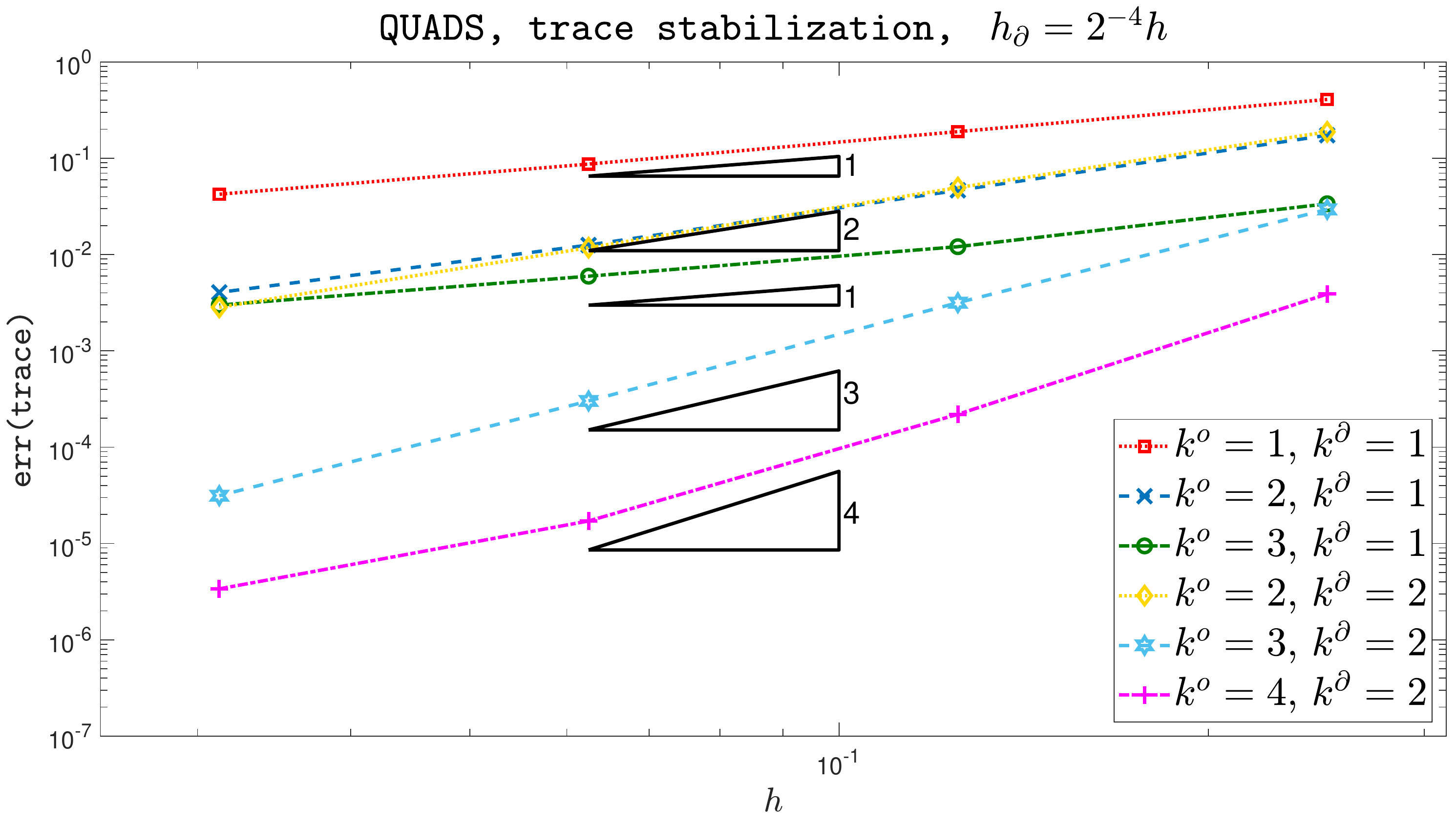}
\caption{Test 1. 
\texttt{QUADS} meshes $\mathcal{Q}_h^{2^{-4}h}$:  \texttt{err(bulk)} (upper) and
\texttt{err(trace)} (lower)  with the
\texttt{dofi-dofi} stabilization (left) and the \texttt{trace} stabilization (right).}
\label{fig:test1_1}
\end{center}
\end{figure}

In Fig. \ref{fig:test1_2} we consider the ``reverse'' point of view, and fix our attention on the Voronoi mesh family: we take the the Voronoi meshes $\mathcal{V}_{2^{-5}}^{\hpartial}$, obtained with the finest diameter $h$, and plot the errors \texttt{err(bulk)} and \texttt{err(trace)} when reducing $\hpartial$. Note that in this investigation the mesh size $h$ (i.e. the element diameters) is not decreasing, as we are only subdividing the element edges into smaller ones. Therefore, as expected, in the case $\kint=\kbou$ there is no error reduction in the graphs. On the other hand, for $\kbou>\kint$ we expect, cf. bound \eqref{paragone}, that the bulk component of the error becomes less relevant and to recover an $O(h_\partial^{\kbou})$ rate of convergence. 
This phenomena can be appreciated in the graphs, expecially for the $\kint=\kbou+2$ cases; clearly, as the edge finesse is increased, the bulk component of the error becomes more relevant and this explains the bends in the curves (for $\hpartial \rightarrow 0$ the error does not converge to zero). Note also that, as expected, the bulk part of the error is more significant for \texttt{err(bulk)} than \texttt{err(trace)}.
Finally we notice that for $\kint=\kbou=1$ the error \texttt{err(trace)} for the \texttt{dofi-dofi} stabilization case is adversely affected by the increasing number of edges $\ell_E$, cf. again \eqref{paragone}. On the contrary, as expected, the error obtained with the \texttt{trace} stabilization is not affected by $\ell_E$.
Analogous results where obtained for the family of quadrilateral meshes $\mathcal{Q}_{2^{-5}}^{\hpartial}$ (not reported).
\begin{figure}[!htb]
\begin{center}
\includegraphics[width=0.49\textwidth]{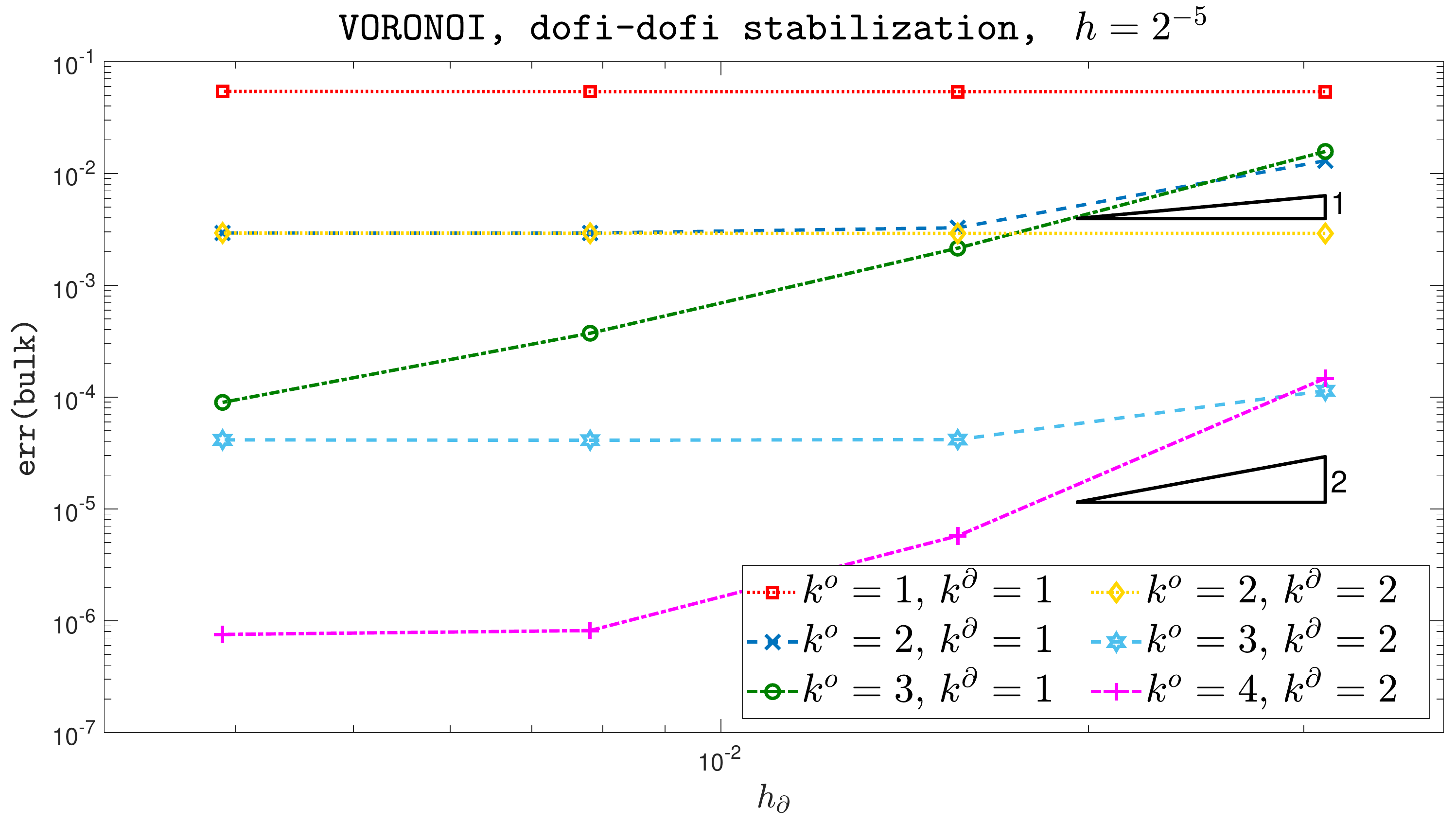}
\includegraphics[width=0.49\textwidth]{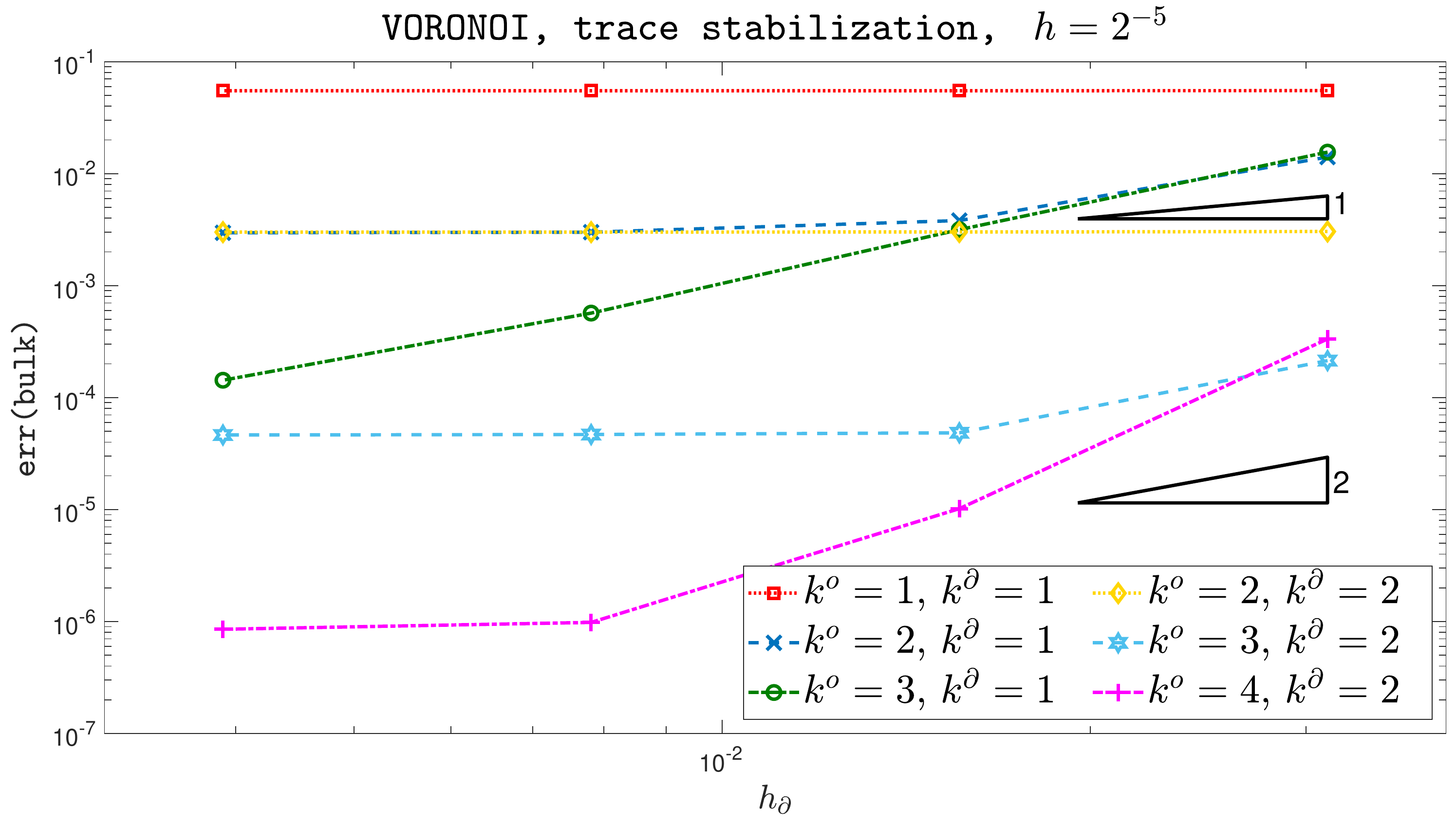}
\\
\includegraphics[width=0.49\textwidth]{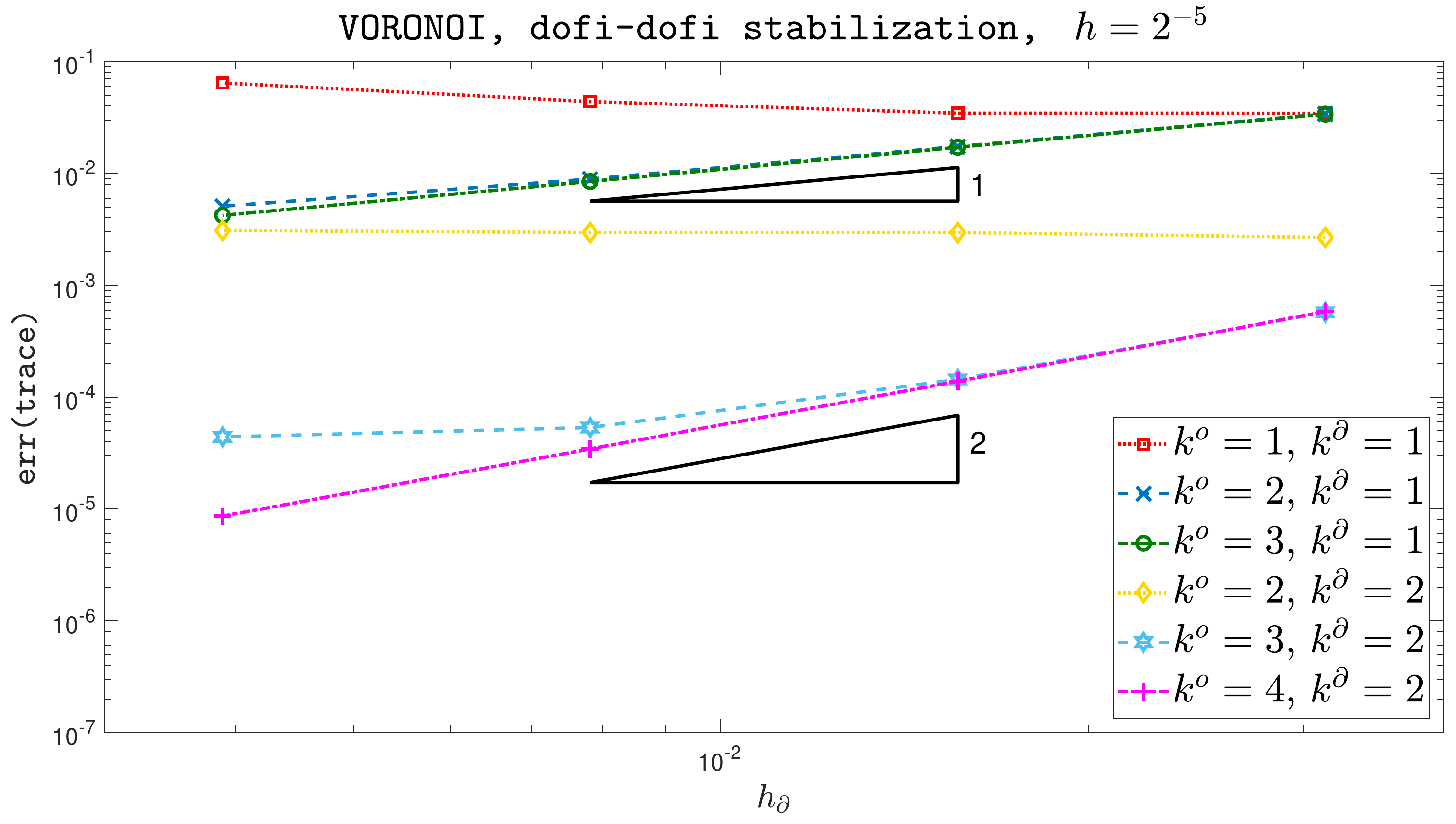}
\includegraphics[width=0.49\textwidth]{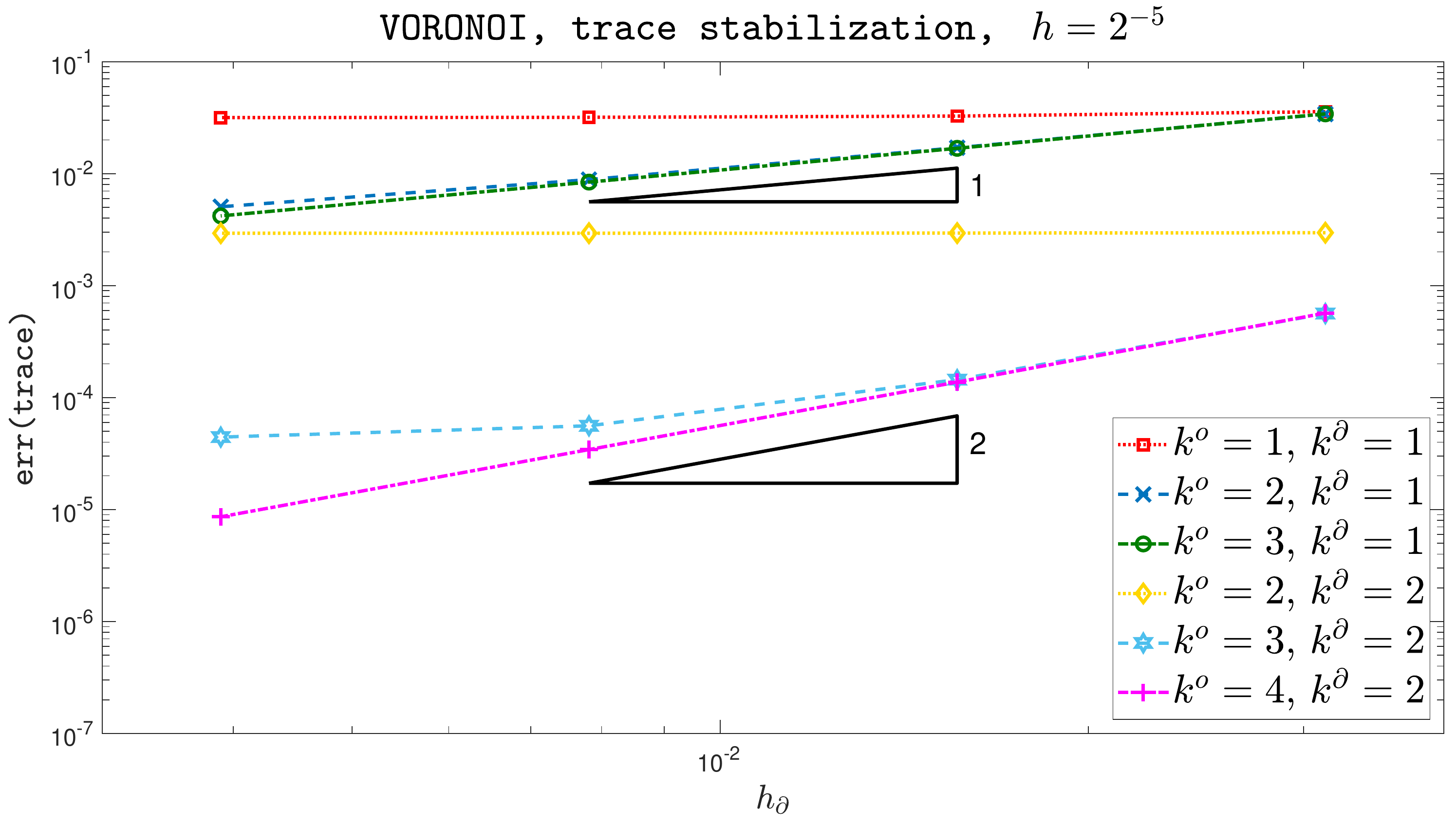}
\end{center}
\caption{Test 1. 
\texttt{VORONOI} meshes $\mathcal{V}_{2^{-5}}^{\hpartial}$:  \texttt{err(bulk)} (upper) and
\texttt{err(trace)} (lower)  with the
\texttt{dofi-dofi} stabilization (left) and the \texttt{trace} stabilization (right).}
\label{fig:test1_2}
\end{figure}

\textbf{Test 2 (Comparison on Voronoi meshes)}
\label{test2}
The goal of the present test is to show the potential advantage of the enriched version in a practical situation. We therefore consider a standard family of Voronoi meshes (namely $\mathcal{V}_h$ of Test 1, without any further subdivision of the edges) and compare the standard VEM $\kint=\kbou$ with the simplest enriched version $\kint=\kbou+1$. 
We stress that the extra DoFs for $\kint > \kbou$ are only internal degrees of freedom and can be easily eliminated from the final linear system by a static condensation procedure; therefore the computational cost of the two schemes is very similar.
In Tab. \ref{tab:test2_1} and Tab. \ref{tab:test2_2} we display respectively the errors \texttt{err(bulk)} and \texttt{err(trace)}  for the generalized VEM scheme with $\kk=(\kbou+1,\kbou)$ and its standard version $\kk=(\kbou,\kbou)$ for $\kbou=1,2$.
In both tables we use the \texttt{dofi-dofi} stabilization, but similar results are obtained with other stabilization options.

\begin{table}[!h]
\centering
\textbf{\texttt{err(bulk)}}
\\
\vspace{0.15cm}
\begin{small}
\begin{tabular}{ccccc}
\toprule 
   $\hbulk$
&  $\kk = (1, 1)$      
&  $\kk = (2, 1)$  
&  $\kk = (2, 2)$      
&  $\kk = (3, 2)$             
\\
\midrule      
\texttt{2\string^-2}
& \texttt{4.5237e-01}
& \texttt{2.7773e-01}
& \texttt{1.7343e-01}
& \texttt{2.3925e-02}
\\
\texttt{2\string^-3}
& \texttt{2.1887e-01}
& \texttt{8.0537e-02}
& \texttt{4.5378e-02}
& \texttt{4.1368e-03} 
\\
\texttt{2\string^-4}
& \texttt{1.1186e-01}
& \texttt{3.2719e-02}
& \texttt{1.1664e-02}
& \texttt{5.3684e-04}
\\
\texttt{2\string^-5}
& \texttt{5.3810e-02}
& \texttt{1.2991e-02}
& \texttt{2.9066e-03}
& \texttt{1.1396e-04}           
\\
\bottomrule
\end{tabular}
\end{small}
\caption{Test 2. \texttt{err(bulk)} for the orders $\kk = (\kint, \kbou)$ with $\kbou =1,2$ and $\kint = \kbou, \kbou+1$ (\texttt{dofi-dofi stabilization}).}
\label{tab:test2_1}
\end{table}

\begin{table}[!h]
\centering
\textbf{\texttt{err(trace)}}
\\
\vspace{0.15cm}
\begin{small}
\begin{tabular}{ccccc}
\toprule 
   $\hbulk$
&  $\kk = (1, 1)$      
&  $\kk = (2, 1)$  
&  $\kk = (2, 2)$      
&  $\kk = (3, 2)$             
\\
\midrule      
\texttt{2\string^-2}
& \texttt{3.8435e-01}
& \texttt{3.7152e-01}
& \texttt{1.5609e-01}
& \texttt{4.3160e-02}
\\
\texttt{2\string^-3}
& \texttt{1.5516e-01}
& \texttt{1.5173e-01}
& \texttt{4.1920e-02}
& \texttt{1.1299e-02} 
\\
\texttt{2\string^-4}
& \texttt{7.4820e-02}
& \texttt{7.3468e-02}
& \texttt{1.0527e-02}
& \texttt{2.1700e-03}
\\
\texttt{2\string^-5}
& \texttt{3.4431e-02}
& \texttt{3.4020e-02}
& \texttt{2.6730e-03}
& \texttt{5.7223e-04}           
\\
\bottomrule
\end{tabular}
\end{small}
\caption{Test 2. \texttt{err(trace)} for the orders $\kk = (\kint, \kbou)$ with $\kbou =1,2$ and $\kint = \kbou, \kbou+1$ (\texttt{dofi-dofi stabilization}).}
\label{tab:test2_2}
\end{table}
For the meshes under considerations $4 \leq \ell_E \leq 8$ therefore the polygons have a moderate number of edges (compared with those in Test 1 and the agglomerated meshes in Test 3), however the benefit provided by the generalized VEM is evident.
The error \texttt{err(bulk)} is reduced in the last refinement of a factor $\approx 4$ for $\kbou=1$ and a factor $\approx 30$ for  $\kbou=2$. It is interesting that, even if the \texttt{err(trace)} is an evaluation of the error on the element boundaries, in the case $\kbou=2$ the enriched version (that we recall modifies the elements only internally) still achieves a significantly better accuracy, roughly a factor of $\approx 4$.

\textbf{Test 3 (Agglomeration meshes)}
\label{test3}
The aim of this test is to consider a family of meshes yielding a more complex element geometry, and to compare the standard VEM ($\kint=\kbou$) with the ``enriched'' one ($\kint>\kbou$).
We consider a sequence of partitions arising from an agglomeration procedure, see for instance \cite{pennesi}, that we depict in Fig. \ref{fig:agglomeration}. 
\begin{figure}[!htb]
\begin{center}
\begin{overpic}[width=0.24\textwidth]{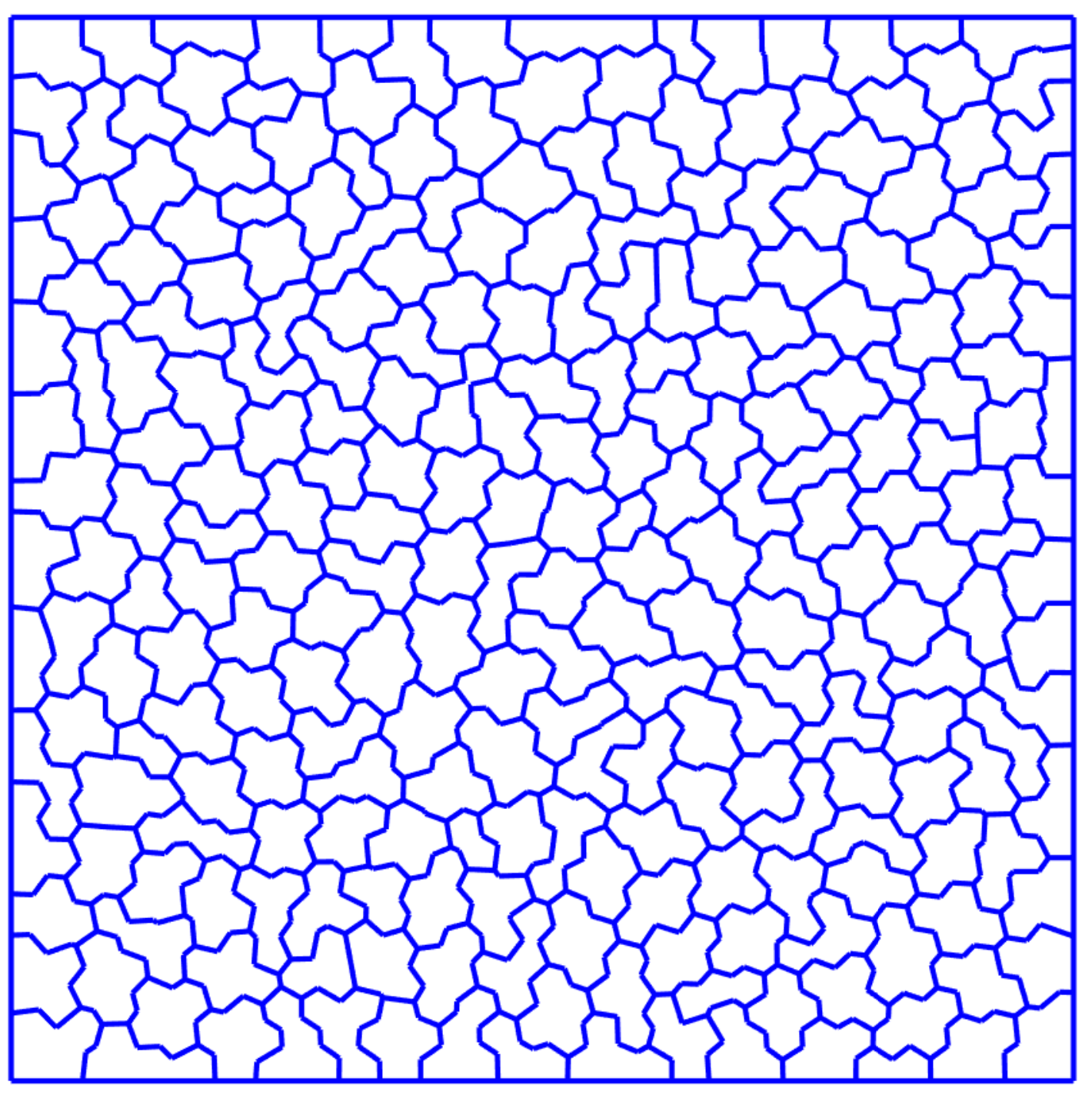}
\put(10, -10){{\texttt{mesh1\_h}, $h \approx 2 h_{\partial}$}}
\end{overpic}
\begin{overpic}[width=0.24\textwidth]{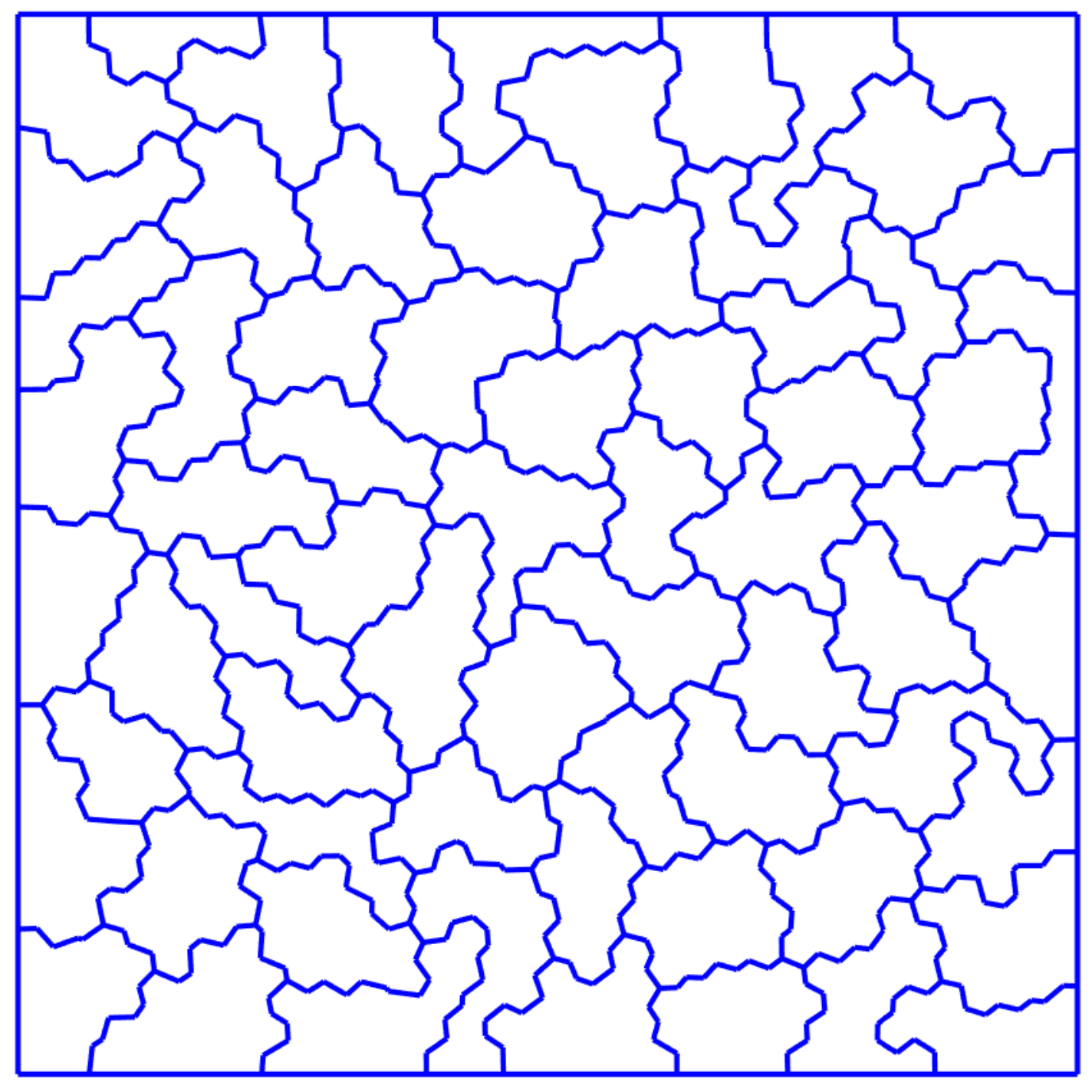}
\put(10, -10){{\texttt{mesh2\_h}, $h \approx 4 h_{\partial}$}}
\end{overpic}
\begin{overpic}[width=0.24\textwidth]{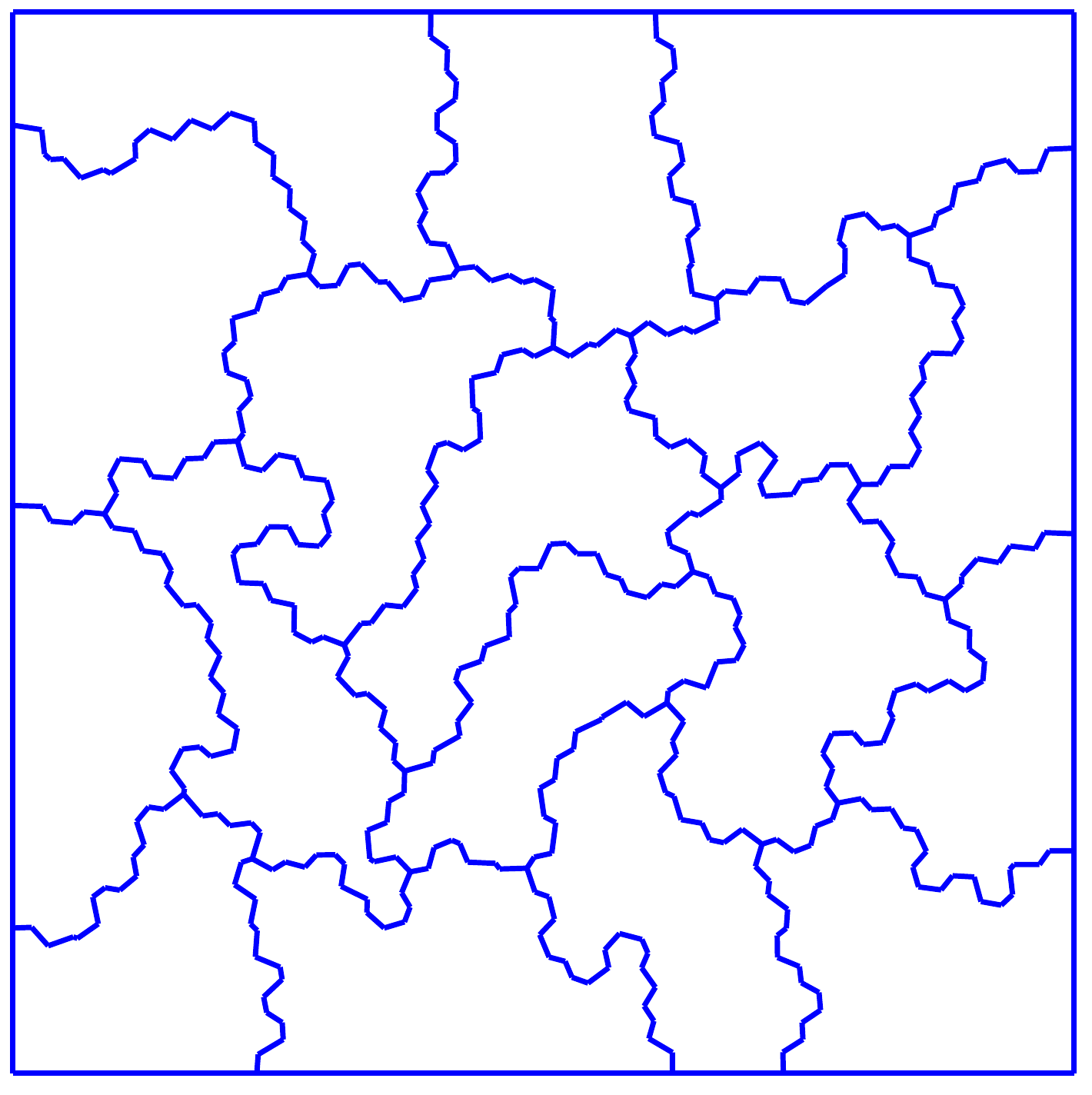}
\put(10, -10){{\texttt{mesh3\_h}, $h \approx 8 h_{\partial}$}}
\end{overpic}
\begin{overpic}[width=0.24\textwidth]{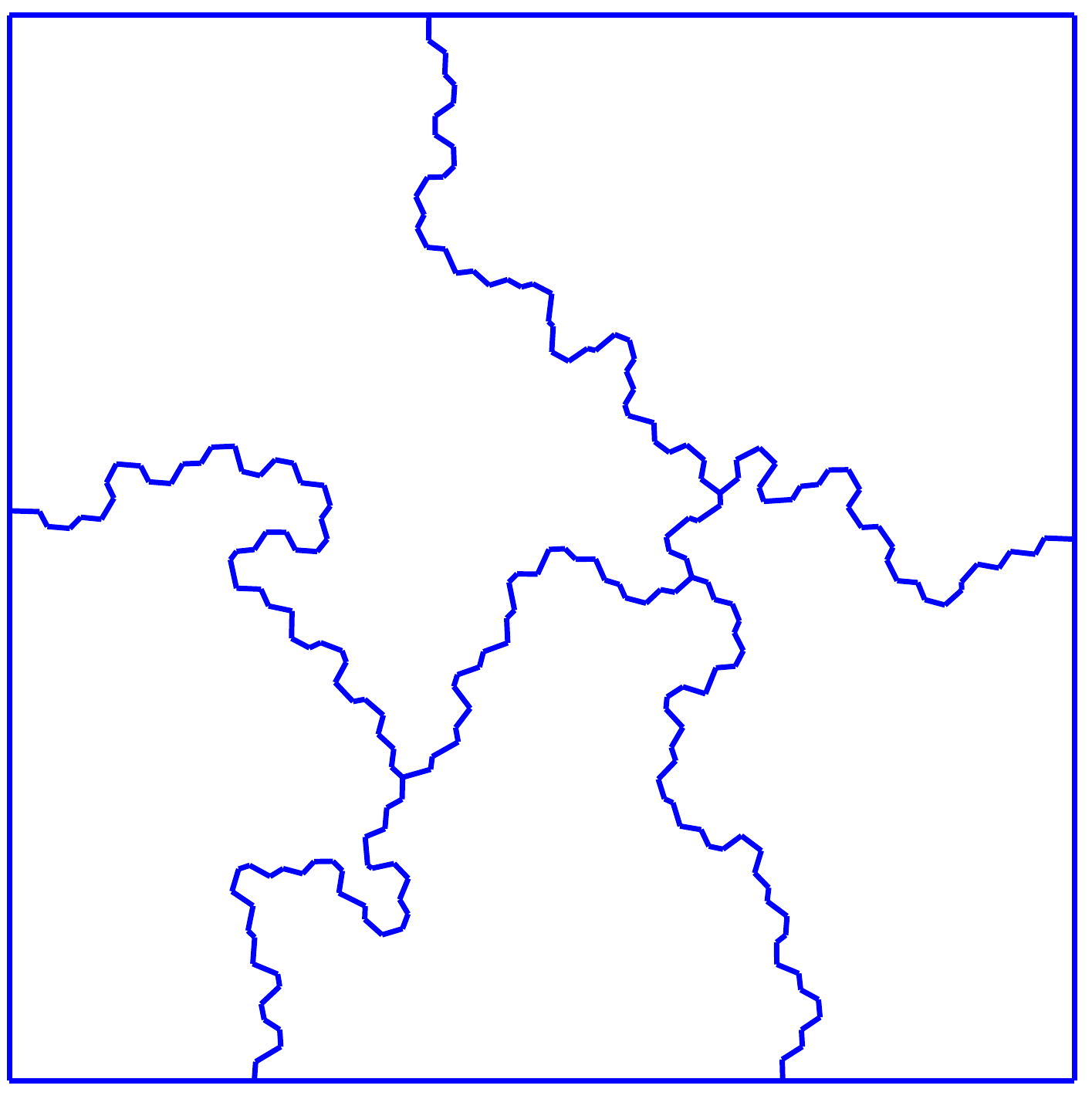}
\put(8, -10){{\texttt{mesh4\_h}, $h \approx 16 h_{\partial}$}}
\end{overpic}
\vspace{0.25cm}
\caption{Test 3. Example of the adopted agglomeration meshes with $h_{\partial} \approx\frac{1}{32}$.}
\label{fig:agglomeration}
\end{center}
\end{figure}
In order to compare the generalized VEM (with $\kint>\kbou$) with its standard counterpart (with $\kint=\kbou$) we define the following 
\[
\texttt{err\%}:= 
\frac{\text{\texttt{err(bulk) with} $\kint > \kbou$}}
{\text{\texttt{err(bulk) with} $\kint = \kbou$}} \,\, \texttt{\%} \,.
\]
%
In Tab. \ref{tab:kbou1} (resp. Tab. \ref{tab:kbou2}) we display the error \texttt{err(bulk)} for the generalized VEM scheme with $\kk=(2,1)$ and $\kk=(3,1)$ (resp.  $\kk=(4,2)$ and $\kk=(4,2)$) and we show the percentage above with respect to the standard VEM scheme of order $k=1$ (resp. $k=2$).
In both tables we use the \texttt{dofi-dofi} stabilization, but similar results can be obtained with other stabilization options.
Agglomerated meshes have very small edges with respect to the element diameter. As a consequence, in the light of our theoretical investigations, we expect the bulk component of the error to be dominant. Therefore, a higher value of $\kint$ is strongly beneficial, as it can be clearly appreciated in the tables (especially in the cases $\kint=\kbou+2$). On the other hand, cf. bound \eqref{paragone}, the gain with respect to the standard case (that is the ratio among the $\kint>\kbou$ case and the standard case in which also the internal degree is taken equal to $\kbou$) is expected to behave as $h^{\kint-\kbou}$; this explains why the percentages are often more favorable for the less agglomerated meshes (which have a smaller $h$).

\begin{table}[!h]
\centering
\begin{small}
\begin{tabular}{clccrcr}
\toprule
& 
& {$\kk=(1,1)$}
& \multicolumn{2}{c}{$\kk=(2,1)$}
& \multicolumn{2}{c}{$\kk=(3,1)$}
\\
\midrule
   $\hpartial$
&  $\texttt{mesh}$      
&  $\texttt{err(bulk)}$ 
&  $\texttt{err(bulk)}$   
&  $\texttt{err\%}$        
&  $\texttt{err(bulk)}$   
&  $\texttt{err\%}$            
\\
\midrule      
\multirow{4}*{$\frac{\sqrt2}{32}$}
& \texttt{1\_h}
& \texttt{1.6963e-01}
& \texttt{2.8258e-02}
& \texttt{16.6\%}
& \texttt{9.6003e-03}
& \texttt{5.6\%} 
\\
& \texttt{2\_h}
& \texttt{3.5495e-01}
& \texttt{1.1075e-01}
& \texttt{31.2\%}
& \texttt{1.4078e-02}
& \texttt{3.9\%}  
\\
& \texttt{3\_h}
& \texttt{7.4933e-01}
& \texttt{5.1930e-01}
& \texttt{69.3\%}
& \texttt{9.3616e-02}
& \texttt{12.4\%}  
\\
& \texttt{4\_h}
& \texttt{9.0279e-01}
& \texttt{7.8457e-01}
& \texttt{86.9\%}
& \texttt{5.1019e-01}
& \texttt{56.5\%}          
\\
\midrule       
\multirow{4}*{$\frac{1}{32}$}
& \texttt{1\_h}
& \texttt{1.1713e-01}
& \texttt{1.4266e-02}
& \texttt{12.1\%}
& \texttt{5.5537e-03}
& \texttt{4.7\%}  
\\
& \texttt{2\_h}
& \texttt{2.5778e-01}
& \texttt{5.3835e-02}
& \texttt{20.8\%}
& \texttt{9.2535e-03}
& \texttt{3.5\%}  
\\
& \texttt{3\_h}
& \texttt{4.9759e-01}
& \texttt{2.0116e-01}
& \texttt{40.4\%}
& \texttt{3.0972e-02}
& \texttt{6.2\%}
\\
& \texttt{4\_h}
& \texttt{7.8435e-01}
& \texttt{4.5238e-01}
& \texttt{57.6\%}
& \texttt{2.0514e-01}
& \texttt{26.1\%}  
\\
\midrule       
\multirow{4}*{$\frac{\sqrt{2}}{64}$}
& \texttt{1\_h}
& \texttt{8.4249e-02}
& \texttt{7.9439e-03}
& \texttt{9.4\%}
& \texttt{4.5417e-03}
& \texttt{5.3\%}  
\\
& \texttt{2\_h}
& \texttt{1.7184e-01}
& \texttt{3.1663e-02}
& \texttt{18.4\%}
& \texttt{9.4749e-03}
& \texttt{5.5\%} 
\\
& \texttt{3\_h}
& \texttt{3.7626e-01}
& \texttt{1.1714e-01}
& \texttt{31.1\%}
& \texttt{2.9892e-02}
& \texttt{7.9\%}
\\
& \texttt{4\_h}
& \texttt{6.6921e-01}
& \texttt{4.2685e-01}
& \texttt{63.7\%}
& \texttt{9.3701e-02}
& \texttt{14.0\%} 
\\
\bottomrule
\end{tabular}
\end{small}
\caption{Test 3. \texttt{err(bulk)} for the orders $\kk = (2,1)$ and $\kk = (3,1)$ 
compared with $\kk=(1,1)$ (\texttt{dofi-dofi stabilization}).}
\label{tab:kbou1}
\end{table}

\begin{table}[!h]
\centering
\begin{small}
\begin{tabular}{clccrcr}
\toprule
& 
& {$\kk=(2,2)$}
& \multicolumn{2}{c}{$\kk=(3,2)$}
& \multicolumn{2}{c}{$\kk=(4,2)$}
\\
\midrule
   $\hpartial$
&  $\texttt{mesh}$      
&  $\texttt{err(bulk)}$ 
&  $\texttt{err(bulk)}$   
&  $\texttt{err\%}$        
&  $\texttt{err(bulk)}$   
&  $\texttt{err\%}$            
\\
\midrule      
\multirow{4}*{$\frac{\sqrt2}{32}$}
& \texttt{1\_h}
& \texttt{2.7234e-02}
& \texttt{1.3951e-03}
& \texttt{5.1\%}
& \texttt{1.4407e-04}
& \texttt{0.5\%} 
\\
& \texttt{2\_h}
& \texttt{1.1032e-01}
& \texttt{9.8795e-03}
& \texttt{8.9\%}
& \texttt{1.7067e-03}
& \texttt{1.5\%} 
\\
& \texttt{3\_h}
& \texttt{5.1839e-01}
& \texttt{9.3871e-02}
& \texttt{18.1\%}
& \texttt{4.2060e-02}
& \texttt{8.1\%}
\\
& \texttt{4\_h}
& \texttt{7.8868e-01}
& \texttt{5.1154e-01}
& \texttt{64.8\%}
& \texttt{2.9915e-01}
& \texttt{37.9\%}            
\\
\midrule       
\multirow{4}*{$\frac{1}{32}$}
& \texttt{1\_h}
& \texttt{1.3555e-02}
& \texttt{4.8473e-04}
& \texttt{3.5\%}
& \texttt{7.1770e-05}
& \texttt{0.5\%} 
\\
& \texttt{2\_h}
& \texttt{5.3251e-02}
& \texttt{4.4397e-03}
& \texttt{8.3\%}
& \texttt{4.4471e-04}
& \texttt{0.8\%}
\\
& \texttt{3\_h}
& \texttt{2.0071e-01}
& \texttt{2.8174e-02}
& \texttt{14.0\%}
& \texttt{6.9751e-03}
& \texttt{3.4\%} 
\\
& \texttt{4\_h}
& \texttt{4.5088e-01}
& \texttt{2.0467e-01}
& \texttt{45.3\%}
& \texttt{4.7336e-02}
& \texttt{10.4\%}
\\
\midrule       
\multirow{4}*{$\frac{\sqrt{2}}{64}$}
& \texttt{1\_h}
& \texttt{7.0567e-03}
& \texttt{1.7948e-04}
& \texttt{2.5\%}
& \texttt{5.4322e-05}
& \texttt{0.7\%}
\\
& \texttt{2\_h}
& \texttt{3.0688e-02}
& \texttt{1.4783e-03}
& \texttt{4.8\%}
& \texttt{3.8698e-04}
& \texttt{1.2\%}
\\
& \texttt{3\_h}
& \texttt{1.1422e-01}
& \texttt{1.2267e-02}
& \texttt{10.7\%}
& \texttt{2.3479e-03}
& \texttt{2.0\%} 
\\
& \texttt{4\_h}
& \texttt{4.2342e-01}
& \texttt{8.1701e-02}
& \texttt{19.2\%}
& \texttt{1.7894e-02}
& \texttt{4.2\%} 
\\
\bottomrule
\end{tabular}
\end{small}
\caption{Test 3. \texttt{err(bulk)} for the orders $\kk = (3,2)$ and $\kk = (4,2)$ 
compared with $\kk=(2,2)$ (\texttt{dofi-dofi stabilization}).}
\label{tab:kbou2}
\end{table}


\subsection*{Appendix}
\def\dx{\textrm{d}x}
\def\dy{\textrm{d}y}

In the present section we give a proof of Lemma \ref{lem:hum}. We first note that, although the norm on the right hand side may seem disproportionately strong, the estimate is sharp also in term of number of edges. Indeed, let $\psi_N$ denote a piecewise linear function on a simple uniform mesh (with $N$ elements) of the interval $[0,1]$, that takes value $1$ on the odd-index nodes and value $-1$ on the even index nodes. Then, it is easy to check that
$$
| \psi_h |_{H^{1/2}(0,1)}^2 \sim N \ , \quad \sum_{e \in {\cal T}_h} \| \psi_h \|_{L^{\infty}(e)}^2 \sim N
\quad \textrm{ as } N \rightarrow \infty \, .
$$
\emph{Proof of Lemma \ref{lem:hum}}.
In the proof, $C$ will denote a generic positive constant, that may change at each occurrence. 
Let ${\cal T}_h$ be a generic mesh of the piecewise quasi-uniform family, associated to an interval $I^h$, and let a generic function $v_h \in {\mathbb S}_k({\cal T}_h)$.  
Let $I^h_n$, for $n=1,2,\dots,\overline{m}$, denote the disjoint sub-intervals associated to the definition of piecewise quasi uniform mesh.
It is clearly not restrictive to assume there are exactly $\overline{m}$ of such subintervals, and it serves the purpose of simplifying the notation. Clearly, the extrema of such sub-intervals are nodes of the mesh ${\cal T}_h$; we define $v_L$ as the unique piecewise linear function (on the mesh ${\cal T}_h$) that takes value $v_h(x)$ on the nodes $x$ that are extrema of a sub-interval, and vanish at all the remaining nodes. 
By following the same direct calculation as in the final part of the proof of Lemma 6.6 in \cite{BLR:2017}, one can easily infer that it exists a constant $C=C(\overline{m})$ such that
\begin{equation}\label{vL-bound}
| v_L |_{1/2,I^h}^2 \leq C \log(1 + R_h) \| v_h\|_{L^\infty(I^h)}^2 \, .
\end{equation}
Moreover, if we define $w_h = v_h - v_L$, such function will vanish at all sub-interval extrema. Therefore, first by a triangle inequality, then using equation (6.12) in \cite{BLR:2017}, it follows
$$
| v_h |_{1/2,I^h} \leq C \left(
\sum_{n=1}^{\overline m} \| w_h \|_{H^{1/2}_{00}(I_h^n)}
+ | v_L |_{1/2,I^h} \right) \, ,
$$
with $C$ universal constant.
Taking the square of the above bound and applying \eqref{vL-bound}, we obtain
\begin{equation}\label{TI-bound}
| v_h |_{1/2,I^h}^2 \leq C \left( 
\sum_{n=1}^{\overline m} \| w_h \|_{H^{1/2}_{00}(I_h^n)}^2
+ \log(1 + R_h) \| v_h\|_{L^\infty(I^h)}^2\right) \, ,
\end{equation}
where now $C=C(\overline{m})$. 
We are left to bound the first term in the right hand side of \eqref{TI-bound}. We fix the attention on a single interval $I_h^n$, $n \in \{1,2,..,\overline{m}\}$, and the associated quasi-uniform mesh, which we denote by $\omega_h$. Let $\{ x_i \}_{i=1}^r$ denote the nodes of $\omega_h$, let $e_i=(x_{i-1}, x_i)$ represent the elements and (by a small abuse of notation) let $h$ represent the characteristic mesh size (cf. the definition of $\overline{c}$ in Definition \ref{pqu}). 
We recall 
\begin{equation}\label{kokoa}
\| w_h \|_{H^{1/2}_{00}(I_h^n)}^2 = 
\int_{I_h^n} \! \int_{I_h^n} \frac{(w_h(x) - w_h(y))^2}{(x-y)^2} \dx\,\dy
+ \int_{I_h^n} \frac{w_h(x)^2}{\rho(x)} \dx \, ,
\end{equation}
where $\rho(x)$ represents the distance of $x$ from the nearest extrema of $I_h^n$.
We here deal only with the first term in the right hand side of the above equation, since the second one can be bounded with analogous arguments. 
By trivial manipulations
\begin{equation}\label{termsplit}
\int_{I_h^n} \!  \int_{I_h^n} \frac{(w_h(x) - w_h(y))^2}{(x-y)^2} \dx\,\dy \leq 2 \, T_1 + T_2
\end{equation}
with
$$
\begin{aligned}
T_1 & = \sum_{i=1}^{r-2} \sum_{j=i+2}^r \int_{e_i} \! \int_{e_j}  \frac{(w_h(x) - w_h(y))^2}{(x-y)^2} \dy\,\dx 
\\
T_2 & =
\sum_{i=1}^{r} \sum_{j=i-1}^{i+1} \int_{e_i} \! \int_{e_j}  \frac{(w_h(x) - w_h(y))^2}{(x-y)^2} \dy\,\dx \, ,
\end{aligned} 
$$
and where, rigorously speaking, the second sum in $T_2$ is 
$\sum_{j=\max{(i-1,1)}}^{\min{(i+1,r)}}$ but we prefer to avoid a heavier notation.
It is easy to check the validity of the following bounds:
$$
\begin{aligned}
T_1 & \leq 2 \sum_{i=1}^{r-2} \sum_{j=i+2}^r \int_{e_i} \! \int_{e_j}  \frac{w_h(x)^2}{(x-y)^2} \dy\,\dx 
+ 2 \sum_{i=1}^{r-2} \sum_{j=i+2}^r \int_{e_i} \! \int_{e_j}  \frac{ w_h(y)^2}{(x-y)^2} 
\dy\,\dx \\
& \leq C \sum_{i=1}^{r-2} \sum_{j=i+2}^r  \| w_h \|_{L^\infty(e_i)}^2  
\int_{e_i} \! \int_{e_j} \frac{1}{(j-i-1)^2 h^2} \dy\,\dx \\
& + C \sum_{i=1}^{r-2} \sum_{j=i+2}^r  \| w_h \|_{L^\infty(e_j)}^2  
\int_{e_i} \! \int_{e_j} \frac{1}{(j-i-1)^2 h^2} \dy\,\dx \\
& \leq C \sum_{i=1}^{r-2} \sum_{j=i+2}^r  \| w_h \|_{L^\infty(e_i)}^2  
\frac{1}{(j-i-1)^2}  
+ C \sum_{i=1}^{r-2} \sum_{j=i+2}^r  \| w_h \|_{L^\infty(e_j)}^2  
\frac{1}{(j-i-1)^2}
\end{aligned}
$$
where the constant $C$ above only depends on $\overline{c}$.
Recalling $\sum_{n=1}^{+\infty} n^{-2} < + \infty$ and rearranging the terms in the sum, the above bound yields
\begin{equation}\label{b:T1}
T_1 \leq C \sum_{i=1}^{r-2} \| w_h \|_{L^\infty(e_i)}^2 +
C \sum_{j=3}^r  \| w_h \|_{L^\infty(e_j)}^2 \leq C \sum_{i=1}^{r} \| w_h \|_{L^\infty(e_i)}^2 \, .
\end{equation}
For the term $T_2$, by the Lipschitz continuity of $w_h$ we infer
$$
T_2 \leq \sum_{i=1}^{r} \sum_{j=i-1}^{i+1} \| w_h' \|_{L^\infty({\tilde{e}}_i)}^2 \int_{e_i} \! \int_{e_j}  1 \ \dy\,\dx ,
$$
where the extended interval $\tilde{e}_i = (x_{i-2},x_{i+1})$ with the usual modification for $i=1$ or $i=r$.
Starting from the above bound, by an inverse estimate for piecewise polynomials
\begin{equation}\label{b:T2}
T_2 \leq C \sum_{i=1}^{r} \sum_{j=i-1}^{i+1} \| w_h' \|_{L^\infty({\tilde{e}}_i)}^2 h^2 
\leq C \sum_{i=1}^{r} \sum_{j=i-1}^{i+1} \| w_h \|_{L^\infty({\tilde{e}}_i)}^2
\leq C \sum_{i=1}^{r} \| w_h \|_{L^\infty(e_i)}^2 \ ,
\end{equation}
where the constant $C$ depends on $k$ and $\overline{c}$.
We now combine \eqref{b:T1} and \eqref{b:T2} into \eqref{termsplit}, then bound the second term of \eqref{kokoa} with analogous arguments. We obtain, for any $n \in \{1,2,..,\overline{m}\}$, 
$$
\| w_h \|_{H^{1/2}_{00}(I_h^n)}^2 \leq  C \sum_{i=1}^{r} \| w_h \|_{L^\infty(e_i)}^2 \, ,
$$
with $C=C(\overline{c},k)$. Substituting the above bound in \eqref{TI-bound} gives
\begin{equation}\label{totti} 
| v_h |_{1/2,I^h}^2 \leq C \left( 
\sum_{e \in {\cal T}_h} \| w_h \|_{L^\infty(e)}^2
+ \log(1 + R_h) \| v_h\|_{L^\infty(I^h)}^2 \right) \, ,
\end{equation}
with $C=C(\overline{m},\overline{c},k)$.
Inequality \eqref{totti} yields our bound, since clearly $\| w_h \|_{L^\infty(e)} \leq 2\| v_h \|_{L^\infty(e)}$ for all $e \in {\cal T}_h$. 
Note that \eqref{totti} would actually imply a stronger bound, where the logarithmic term only multiplies the global $L^\infty$ norm.

\qed

\section*{Acknowledgements}
The authors thank P.F. Antonietti and G. Pennesi for providing the agglomerated mesh files used in the numerical tests section.
The authors  were partially supported by the European Research Council through
the H2020 Consolidator Grant (grant no. 681162) CAVE, ``Challenges and Advancements in Virtual Elements''. This support is gratefully acknowledged.
The first author was partially supported by the italian PRIN 2017 grant
``Virtual Element Methods: Analysis and Applications''. This support is gratefully acknowledged.

\addcontentsline{toc}{section}{\refname}
\bibliographystyle{plain}
\bibliography{biblio}

\end{document}